\documentclass[11pt,a4paper]{article}
\usepackage{amsmath,amssymb,amsthm}
\usepackage{mathtools}
\usepackage{fullpage,color}
\usepackage[colorlinks=true,anchorcolor=blue,filecolor=blue,linkcolor=red,urlcolor=blue,citecolor=blue]{hyperref}     
\usepackage{tikz,bm}
\usetikzlibrary{shapes.geometric, arrows.meta, positioning, calc} 
\usepackage{geometry} 
\newtheorem{theorem}{Theorem}[section]
\newtheorem{conjecture}[theorem]{Conjecture}
\newtheorem{lemma}[theorem]{Lemma}
\newtheorem{problem}[theorem]{Problem}
\newtheorem{corollary}[theorem]{Corollary}
\newtheorem{claim}[theorem]{Claim}

\theoremstyle{definition}

\newtheorem{example}[theorem]{Example}

\newtheorem{remark}{Remark}
\theoremstyle{plain}

\newcommand{\tr}{\operatorname{tr}}

\newcommand{\bk}{\operatorname{bk}}

\renewcommand{\leq}{\leqslant}
\renewcommand{\le}{\leqslant}
\renewcommand{\geq}{\geqslant}
\renewcommand{\ge}{\geqslant}

\title{Supersaturation in Nosal graphs: Triangles and books} 

\date{}

\begin{document}

\author{
Hongzhang Chen\thanks{School of Mathematics and Statistics, 
Gansu Center for Applied Mathematics, 
Lanzhou University, Lanzhou, Gansu, 730000, China. Email: \url{mnhzchern@gmail.com}.} 
\and 
Yongtao Li\thanks{Corresponding author. Yau Mathematical Sciences Center, Tsinghua University, Beijing, 100084, China. 
Email: \url{ytli0921@hnu.edu.cn}.}
\and 
Quanyu Tang\thanks{School of Mathematics and Statistics, Xi'an Jiaotong University, Xi'an 710049, China.
Email: \url{tangquanyu827@gmail.com}.}
}

\maketitle

\vspace{-0.4cm}

\begin{abstract}
A graph $G$ with $m$ edges is called \emph{Nosal} if its spectral radius satisfies 
$\lambda(G) > \sqrt{m}$. The spectral assumption $\lambda(G) > \sqrt{m}$ applies to 
 graphs of every edge density, including very sparse ones, and is therefore well suited to edge-spectral extremal problems. In this paper, we use the spectral surplus 
$\lambda(G) - \sqrt{m}$ to measure how far $G$ lies above the Nosal threshold, and 
prove the following edge-spectral supersaturation results for triangles and books.

\begin{itemize}
    \item Every graph $G$ with $m\ge 3$ edges and $\lambda(G) \ge 1 + \sqrt{m-2}$ contains at least $m-2$ triangles, with equality if and only if $G = K_3 \vee \tfrac{m-3}{3} K_1$. 
    This can be viewed as the third-layer supersaturation in the jump phenomenon, after the first layer $t(G) \ge \lfloor \tfrac{1}{2}(\sqrt{m}-1) \rfloor$ proved by Ning and Zhai, and the second layer $t(G) \ge \tfrac{m-1}{2}$ by Zhang and Zhai.

    \item Every $m$-edge graph $G$ satisfies 
    $t(G) \ge m\bigl(\lambda - \sqrt{m}\,\bigr)$, with equality if and only if $G$ 
    is complete bipartite. Consequently, $\lambda(G) \ge \sqrt{m} + q$ forces 
    $t(G) > q m$ for every real $q > 0$. This is an edge-spectral counterpart of the 
    Lov\'asz--Simonovits theorem, and it improves the Bollob\'as--Nikiforov bound 
    $t(G) \ge \tfrac13 \lambda(\lambda^2 - m)$ in the range 
$\sqrt m \le \lambda(G) \le 1.3\sqrt m $.

    \item Every $m$-edge Nosal graph $G$ contains a book of size greater than 
    $\tfrac14 \sqrt{m}$. This improves two recent results on the booksize constant: $\tfrac{1}{24}$ proved by Li, Liu and Zhang, and $\tfrac19$ by Zhai, Li and Lou. 
    This narrows the gap toward the conjectured optimal constant $\tfrac13$. 

    \item Every $m$-edge Nosal graph $G$ contains at least 
    $\bigl(\tfrac{1}{8} - o(1)\bigr) m$ copies of the kite $C_4^+=B_2$, and the constant 
    $\tfrac18$ is best possible. 
    This determines the sharp asymptotic constant for counting $C_4^+$ and strengthens the $\Omega(m)$ bound of Li, Liu and Zhang.
\end{itemize} 
The proofs for triangles are based on a third-moment count that controls $\sum_{i \ge 2} \lambda_i^3$ through the Cauchy interlacing theorem and the power-mean monotonicity. 
The proofs for books are based on a unified framework: a Perron-vector decomposition around a 
vertex of maximum coordinate, an exact identity for the spectral surplus 
$\lambda^2 - m$, and a half-missing-weights estimate. 
\end{abstract}



\section{Introduction}

Let $G$ be a finite simple graph. We denote by $m$ the number of edges of $G$ and by $n$ the number of vertices of $G$. 
Let $A(G)$ be its adjacency matrix and let $\lambda(G)$ be the largest eigenvalue of $A(G)$. 
The isolated vertices are ignored in all extremal and equality statements, since they do not affect the spectral radius and the number of triangles. 
Spectral extremal graph theory investigates how spectral conditions force the existence of certain subgraphs. 
Nosal's theorem \cite{Nosal1970} states that every graph $G$ with $m$ edges and $\lambda(G) > \sqrt{m}$ contains a triangle. 
Since $\lambda (G)$ is at least the average degree $\frac{2m}{n}$, 
Nosal's theorem is a spectral strengthening of Mantel's theorem, which asserts that any $n$-vertex graph with $m > \lfloor n^2/4 \rfloor$ edges must contain a triangle. 
The spectral assumption $\lambda (G)> \sqrt{m}$ is more versatile than the traditional size assumption 
since it not only implies the size version, but can also be applied to sparse graphs with every edge density. For example, the split graph $G=K_2\vee kK_1$ with $k=\frac{m-1}{2}$ satisfies $\lambda (G) > \sqrt{m}$ but $m\ll n^2/4$. 

\smallskip 
Classical supersaturation results usually measure the excess in terms of the number of vertices, e.g., the Moon--Moser theorem \cite{MM1962} (see also \cite[page 443]{Lov1979}) asserts that any $n$-vertex 
graph with $m> n^2/4$ edges contains at least
$\frac{4m}{3n}\big(m- \frac{n^2}{4} \big)$ triangles. This result is effective when $G$ is close to having $n^2/4$ edges, but it gives no information for sparse graphs. The spectral radius gives a new scale that uses only the number of
edges: $\sqrt{m}$ plays the role of $ n^2/4 $ and  stays useful at every edge density, which is why a graph as sparse as $K_2\vee kK_1$ still
falls within reach. {\it Our main motivation is to study supersaturation in this setting for sparse graphs with spectral radius above $\sqrt{m}$.}  

\smallskip 
A graph $G$ with $m$ edges is called {\it Nosal} if its spectral radius satisfies $\lambda (G)> \sqrt{m}$; this concept was formally introduced in \cite{LiLiuZhang26}.  
There are many extremal results for Nosal graphs. Nikiforov \cite{Niki2009laa} proved that every Nosal graph with $m\ge 10$ edges contains a $C_4$; see \cite{LNW2021, ZS2022dm, LLZ2025, WL2026} for related generalizations. Ning and Zhai \cite{NZ2021b} showed that $G$ contains at least $\frac{1}{2000}m^2$ copies of $C_4$. Later, Li, Liu and Zhang \cite{LiLiuZhang26} sharpened this result by proving that $G$ contains  at least $(\frac{1}{8} - o(1))m^2$ copies of $C_4$, and the constant $\frac{1}{8}$ is best possible. 
Moreover, they \cite{LiLiuZhang26} also proved that the maximum degree of any Nosal graph $G$ satisfies $\Delta (G)\le (\frac{1}{2} + o(1))m$, and the constant $\frac{1}{2}$ is  best possible. 
In the study of edge-spectral stability, Li, Liu and Zhang \cite{LLZ-edge-spectral} proved that every Nosal graph contains a copy of $K_{2,t}$ with $t\ge \frac{1}{2} \sqrt{m} - O(1)$, and this bound is best possible up to an $O(1)$ term.

\smallskip 
In this paper, 
we investigate the supersaturation as $\lambda(G)$ grows past Nosal's bound $\sqrt{m}$. Let $t(G)$ denote the number
of triangles in $G$. 
Continuing the first two layers in \cite{NingZhai2023, ZZ2026}, 
we establish the third-layer
supersaturation, with the extremal graph at each jump being a split graph (Theorem~\ref{thm:m-2}). Furthermore, we prove that if $q>0$ and $\lambda (G)\ge \sqrt{m} +q$, then $t(G)> qm$ (Theorem~\ref{thm-edge-spect-LS}). This bound follows from a linear dependence $t(G)\ge m\big(\lambda(G)-\sqrt{m}\,\big)$, with equality if and only if $G$ is complete bipartite (Theorem \ref{thm:strong}). Moreover, we develop a  framework for the problem involving books in Nosal graphs, and within this framework we prove the following two results. First, many triangles can be made to share a common edge, yielding a large book of size greater than $\frac{1}{4}\sqrt{m}$ (Theorems~\ref{thm-size-14} and \ref{thm:book-free}). Second, every Nosal graph contains at least $\big(\frac{1}{8} - o(1)\big)m$ copies of $C_4^+$, the $4$-cycle with an added chord, and the constant $\frac{1}{8}$ is optimal (Theorem~\ref{thm-C4+}). 
Moreover, a general result for counting the book $B_t$ unifies the cases of triangles and kites (Theorem \ref{thm:Bt-count}).

\subsection{Jump phenomenon for counting triangles}

Erd\H{o}s and Rademacher (see, e.g.,  \cite{Erd1955,Erdos1964}) extended the Mantel theorem by showing that  
if $G$ is a graph on $n$ vertices with 
$m> \lfloor n^2/4 \rfloor$, 
then $t(G) \ge \lfloor {n}/{2}\rfloor$; we refer to \cite{LPS2020,XK2021,LM2022-Erd-Rad,BC2023} for recent generalizations.  
In 2007, Bollob\'as and Nikiforov
\cite{BollobasNikiforov2007} initiated the study of spectral supersaturation and showed that 
if $G$ is an $m$-edge graph with spectral radius $\lambda$, then 
\begin{equation} 
\label{eq:BN}
  t(G) \ge \frac{\lambda \bigl(\lambda^2 - m\bigr)}{3}. 
  \end{equation}
  This bound was independently proved by 
  Cioab\u{a}, Feng, Tait and Zhang \cite[Lemma 7]{CFTZ20} in a different form: $m \ge \lambda^2- \frac{3}{\lambda}t(G)$. 
Ning and Zhai \cite{NingZhai2023} showed that the equality holds in (\ref{eq:BN}) if and only if $G$ is a complete bipartite graph.  
Moreover, they proved that if $\lambda(G)\geq \sqrt m$, then $t(G)\ge \lfloor \frac{1}{2} (\sqrt{m}-1) \rfloor$, unless  $G$ is complete bipartite;  
we refer to this as the {\it first-layer} supersaturation.

\smallskip 
An interesting problem is how to establish a general edge-spectral supersaturation when an $m$-edge graph has spectral radius exceeding $\sqrt{m}$. 
As a warm-up, we see from (\ref{eq:BN}) 
that if $\lambda (G)\ge \sqrt{m} + \frac{1}{2}$, then $t(G)\ge \left( \frac{1}{3} + o(1) \right)m$.  
Improving this bound, 
Zhang and Zhai \cite{ZZ2026} recently proved the following sharp supersaturation: if $G$ is an $m$-edge graph with $\lambda (G)\ge \frac{1}{2} (1+ \sqrt{4m-3}\,)$, then 
$ t(G) \ge \frac{m-1}{2}$,  
with equality if and only if 
$G=K_2\vee \frac{m-1}{2}K_1$. This can be viewed as the {\it second-layer} supersaturation.  
In this paper, we establish the {\it third-layer}   supersaturation as follows. 

\begin{theorem} \label{thm:m-2}
Let $G$ be a graph with $m\ge 3$ edges.  If
$\lambda(G)\ge 1+\sqrt{m-2}$, 
then 
\[
        t(G)\ge m-2.
\]
Moreover, the equality holds if and only if
       $G=K_3\vee \frac{m-3}{3}K_1$. 
\end{theorem}

 \begin{remark} 
A heuristic explanation of these jump phenomena is structural:  
In order to increase the spectral radius 
further above $\sqrt{m}$, 
the extremal graph builds a larger clique into its core, 
and since the clique size is an integer, the number of triangles jumps in discrete steps: the first-layer forces $t(G)\approx \tfrac{1}{2}\sqrt{m}$; the second-layer forces $t(G)\approx \tfrac{1}{2}m$, which is extremized by $K_2\vee \tfrac{m-1}{2}K_1$; and the third-layer forces $t(G)\approx m$, which is extremized by $K_3\vee \tfrac{m-3}{3}K_1$. These results fit into a  layer-by-layer pattern, and   
it is meaningful to investigate the $k$-th 
layer for any $k$; see Figure \ref{fig:jump-layers}.  
 \end{remark}

 \vspace{-4mm}
 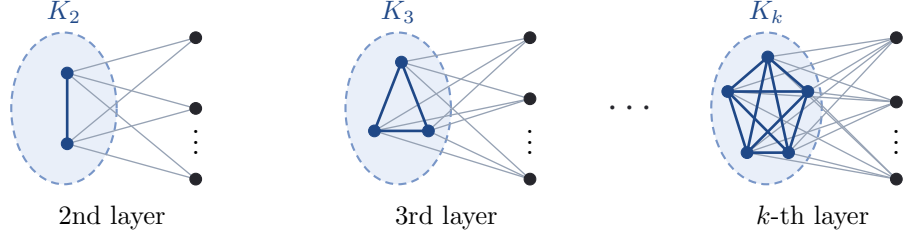
\begin{figure}[htbp]
\centering
\definecolor{cliquecol}{RGB}{31,72,135}
\definecolor{indepcol}{RGB}{38,38,44}
\definecolor{joincol}{RGB}{150,162,180}
\definecolor{blobfill}{RGB}{233,239,249}
\definecolor{blobdraw}{RGB}{120,150,200}

\begin{tikzpicture}[
  cv/.style={circle, fill=cliquecol, draw=cliquecol, minimum size=4.4pt, inner sep=0pt},
  iv/.style={circle, fill=indepcol, draw=indepcol, minimum size=4.4pt, inner sep=0pt},
  cedge/.style={draw=cliquecol, line width=1.0pt},
  jedge/.style={draw=joincol, line width=0.5pt},
  blob/.style={draw=blobdraw, line width=0.8pt, fill=blobfill, dash pattern=on 3pt off 2pt},
  lay/.style={font=\small},
  klab/.style={font=\small\bfseries, text=cliquecol},scale=0.85
]
\begin{scope}[shift={(0,0)}]
  \fill[blob] (-0.05,0) ellipse (0.82 and 1.18);
  \coordinate (Ac1) at (0,0.55);   
  \coordinate (Ac2) at (0,-0.55);
  \coordinate (Aw1) at (2.0,1.1);  
  \coordinate (Aw2) at (2.0,0);  
  \coordinate (Aw3) at (2.0,-1.1);
  \foreach \c in {Ac1,Ac2}\foreach \w in {Aw1,Aw2,Aw3}{\draw[jedge] (\c)--(\w);}
  \draw[cedge] (Ac1)--(Ac2);
  \node[cv] at (Ac1){}; \node[cv] at (Ac2){};
  \node[iv] at (Aw1){}; \node[iv] at (Aw2){}; \node[iv] at (Aw3){};
  \node at (2.0,-0.4){$\vdots$};
  \node[klab] at (-0.05,1.5){$K_2$};
  \node[lay] at (0.7,-1.7){$2$nd layer};
\end{scope}
\begin{scope}[shift={(5.2,0)}]
  \fill[blob] (-0.05,0) ellipse (0.82 and 1.18);
  \coordinate (Bc1) at (0,0.72); 
  \coordinate (Bc2) at (-0.42,-0.35); 
  \coordinate (Bc3) at (0.42,-0.35);
  \coordinate (Bw1) at (2.0,1.1); 
  \coordinate (Bw2) at (2.0,0.15);
   \coordinate (Bw3) at (2.0,-1.1);
  \foreach \c in {Bc1,Bc2,Bc3}\foreach \w in {Bw1,Bw2,Bw3}{\draw[jedge] (\c)--(\w);}
  \draw[cedge] (Bc1)--(Bc2)--(Bc3)--(Bc1);
  \node[cv] at (Bc1){}; \node[cv] at (Bc2){}; \node[cv] at (Bc3){};
  \node[iv] at (Bw1){}; \node[iv] at (Bw2){}; \node[iv] at (Bw3){};
  \node at (2.0,-0.4){$\vdots$};
  \node[klab] at (-0.05,1.5){$K_3$};
  \node[lay] at (0.7,-1.7){$3$rd layer};
\end{scope}
\node[font=\Large, text=indepcol] at (8.8,0){$\cdots$};
\begin{scope}[shift={(10.9,0)}]
  \fill[blob] (-0.02,0) ellipse (0.86 and 1.18);
  \coordinate (Cc1) at (0,0.8);    
  \coordinate (Cc2) at (-0.623,0.263);
  \coordinate (Cc3) at (-0.323,-0.688); 
  \coordinate (Cc4) at (0.323,-0.688);
  \coordinate (Cc5) at (0.623,0.263);
  \coordinate (Cw1) at (2.0,1.1); 
  \coordinate (Cw2) at (2.0,0.1); 
  \coordinate (Cw3) at (2.0,-1.1);
  \foreach \c in {Cc1,Cc2,Cc3,Cc4,Cc5}\foreach \w in {Cw1,Cw2,Cw3}{\draw[jedge] (\c)--(\w);}
  \foreach \i/\j in {1/2,1/3,1/4,1/5,2/3,2/4,2/5,3/4,3/5,4/5}{\draw[cedge] (Cc\i)--(Cc\j);}
  \foreach \i in {1,2,3,4,5}{\node[cv] at (Cc\i){};}
  \node[iv] at (Cw1){}; \node[iv] at (Cw2){}; \node[iv] at (Cw3){};
  \node at (2.0,-0.4){$\vdots$};
  \node[klab] at (-0.02,1.5){$K_k$};
  \node[lay] at (0.7,-1.7){$k$-th layer};
\end{scope}
\end{tikzpicture} 
\vspace{-2mm}
\caption{The layer-by-layer jump phenomenon with extremal split graphs.}
\label{fig:jump-layers}
\end{figure}

\vspace{-3mm}
\subsection{An edge-spectral Lov\'{a}sz--Simonovits theorem}

A classical theorem of Lov\'asz and Simonovits \cite{LovaszSimonovits1975,LS1983}
gives a sharp supersaturation for triangles, showing that for any positive integer $q< n/2$, if $G$ is an $n$-vertex graph with size $m\ge \lfloor n^2/4 \rfloor +q$, then $t(G)\ge q\lfloor n/2 \rfloor$; see \cite{Mubayi2010, PY2017, MY2023} for related extensions.   
Intuitively, once an $n$-vertex graph has more edges than the Tur\'{a}n number $\lfloor n^2/4\rfloor$, each additional edge forces at least $\lfloor n/2\rfloor$ additional triangles; this growth accumulates linearly, and the coefficient $\lfloor n/2\rfloor$ is sharp. 
The vertex-spectral counterparts of such triangle counting
results have been studied intensively; see, e.g., \cite{LiFengPeng2026, FLLM2025}. 
It is natural to ask whether the triangle count has a linear accumulation by the spectral surplus $\lambda(G)-\sqrt{m}$. 
The second main result of our paper answers this question and presents an edge-spectral counterpart of the Lov\'{a}sz--Simonovits supersaturation theorem.

\begin{theorem}  \label{thm-edge-spect-LS}
For any real $q> 0$, if $G$ is an $m$-edge graph with $\lambda (G) \ge \sqrt{m}+ q$, then 
$$t(G)> q\cdot m. $$ 
This bound is asymptotically tight as witnessed by the split graphs.
\end{theorem}

  We point out that Theorem \ref{thm-edge-spect-LS} is applicable to sparse graphs, e.g., the split graphs, whereas the Lov\'{a}sz--Simonovits theorem is only applicable to those graphs with nearly quadratic edge density.  
Incidentally, one should not expect an exact equality case in Theorem \ref{thm-edge-spect-LS} for a general real $q>0$. 
Unlike the number of edge, which takes integer values, the spectral
radius $\lambda(G)$ is the largest root of the characteristic polynomial of the adjacency matrix, and it is usually irrational. Hence, a threshold of the form $\sqrt{m}+q$ cannot in general be exactly achieved by the spectral radius of a graph; it can only be approximated. This partly explains why Theorem~\ref{thm-edge-spect-LS} is stated in an 
asymptotically tight form, in contrast with the exact Lov\'{a}sz--Simonovits theorem.  

\smallskip 
We now state a stronger version of Theorem \ref{thm-edge-spect-LS} as follows. 

\begin{theorem} \label{thm:strong}
Let $G$ be an $m$-edge graph with spectral radius $\lambda$.   
Then
\[
    t(G)\geq m\bigl(\lambda -\sqrt m\, \bigr),
\]
with equality if and only if $G$ is a complete bipartite graph. 
\end{theorem}

\begin{remark}   
We compare Theorem~\ref{thm:strong} with the Bollob\'{a}s--Nikiforov bound (\ref{eq:BN}). Both contain the factor $\lambda-\sqrt m$, 
with remaining coefficients $m$ and
$\frac13\lambda(\lambda+\sqrt m\,)$, respectively. 
Upon computation, Theorem~\ref{thm:strong} gives a better bound in the range $ \sqrt m\le \lambda\le \frac{\sqrt{13}-1}{2}\sqrt m\approx 1.3\sqrt{m}$, 
and the Bollob\'{a}s--Nikiforov bound is better in the remaining range $1.3\sqrt{m} < \lambda <  \sqrt{2m}$. 
The coefficient $\tfrac13\lambda(\lambda+\sqrt m\,)$
decays rapidly as $\lambda$ decreases from $\sqrt{2m}$, whereas our linear
coefficient $m$ is independent of $\lambda$ and hence remains effective in the
low spectral radius regime. 
In this regime $\lambda=\sqrt m+q$ with $q=o(\sqrt m\,)$, which can be realized by the split graphs,  
 Theorem \ref{thm:strong} gives the tight bound  $t(G)\ge qm$, 
improving the Bollob\'{a}s--Nikiforov bound  
$t(G)\ge \big(\frac23+o(1)\big)qm$ by a factor of $\frac{3}{2}$.  
\end{remark}

\subsection{Finding a large book in Nosal graphs}

A \emph{book of size $k$} is a graph $B_k$ consisting of $k$ triangles sharing a common edge, and $\bk(G)$ denotes the largest $k$ such that $G$ contains a copy of $B_k$. 
In 1962, Erd\H{o}s \cite{Erd1962} conjectured that any $n$-vertex graph $G$ with $m>n^2/4$ satisfies \( \bk (G) >  n/6 \). 
This was later confirmed by Edwards \cite{Edw1977} and independently by Khad\v{z}iivanov and Nikiforov  \cite{KN1979}. 
The manuscript of Edwards was never published, and the paper of Khad\v{z}iivanov and Nikiforov is not easily accessible;  
two different proofs of this result were  provided by Bollob\'{a}s and Nikiforov \cite{BN05}, and  Li, Feng and Peng \cite{LFP2024-triangular}. 
A large book is a much stronger conclusion than a single triangle: while Nosal's theorem only
guarantees one triangle, a book of size $q$ shows that $q$ triangles sit on a single
edge, which reveals a locally dense substructure of a graph; we refer the readers to \cite{Mubayi2012,CFS2020,ZL2022jgt,LiuMiao2025,MiaoLiuDam2026}. 

\smallskip  
The problem of forcing a large book under the spectral condition was initially raised by Zhai, Lin and Shu \cite{ZLS2021} as a variant of their study of  $K_{2,t}$, and then investigated by Nikiforov \cite{Nikiforov21}.  
First, it is interesting because the spectral statement, in its conjectured form $\bk (G)> \frac13\sqrt{m}$ stated in \cite{LiLiuZhang26}, 
would imply the classical result $\bk(G)>\frac16 n$ for every graph with $m>\frac{1}{4} n^2$, 
which is the Edwards/Khad\v{z}ivanov--Nikiforov theorem. Thus, the conjectured spectral bound would strengthen the classical one. 
Second, the conjectured constant $\frac13$ is best possible. Indeed, we define $C_3^{\square}$ to be the triangular prism, consisting of two vertex-disjoint triangles joined by a perfect matching. Let $G$ be the blow-up of $C_3^{\square}$ obtained by replacing each vertex of the upper triangle with an independent set of size $k+1$, each vertex of the lower triangle with an independent set of size $k-1$,
and each edge with a complete bipartite graph. Then $G$ is of order $n=6k$, and size $m=9k^2+3$, while $\bk(G)=k+1< \frac13\sqrt{m}+1$. The graph $G$ is 
Nosal since $\lambda (G)\ge \frac{2m}{n} =3k + \frac{1}{k} > \sqrt{m}$.  
Motivated by this sharp construction, Li, Liu
and Zhang \cite{LiLiuZhang26} proposed the following problem.

\begin{problem}[Li--Liu--Zhang \cite{LiLiuZhang26}] \label{prob-bk-13}
Does every $m$-edge Nosal graph have a book of size $\frac{1}{3}\sqrt{m}$? 
\end{problem}

Progress on Problem \ref{prob-bk-13} 
comes from improving the coefficient. 
Using the light-edge deletion method, 
Nikiforov \cite{Nikiforov21} proved $\bk(G)>\frac1{12}m^{1/4}$. 
The order was raised to the tight 
$\Omega (\sqrt{m})$ scale by Li, Liu and Zhang \cite{LiLiuZhang26}, who obtained   $\bk(G)>\frac1{24}\sqrt{m}$ by using a probabilistic argument. 
Zhai, Li and Lou \cite{ZhaiLiLou26} improved the constant to $\frac19$.  
Our third main result improves the constant to $\frac{1}{4}$. 

\begin{theorem} \label{thm-size-14}
Every $m$-edge Nosal graph $G$ satisfies
$$ \bk(G)>\frac14\sqrt{m}. $$
\end{theorem}

As an application of Theorem \ref{thm-size-14}, 
we turn to the problem of counting  triangular edges. 
We say that an edge is \emph{triangular} if it lies in a triangle.
Extending Mantel's theorem, Erd\H{o}s, Faudree and Rousseau \cite{EFR1992}
proved that every $n$-vertex graph with more than $\lfloor n^2/4\rfloor$ edges
contains at least $2\lfloor n/2\rfloor +1$ triangular edges. We refer to \cite{FM2017, GL2018, GHV2019} for related results. 
Recently, Li, Feng and Peng \cite{LFP2024-triangular} studied the spectral 
problem of counting triangular edges and established a spectral version of
the Erd\H{o}s--Faudree--Rousseau theorem. Moreover, they proposed the
following conjecture.

\begin{conjecture}[Li--Feng--Peng \cite{LFP2024-triangular}]
\label{thm-main2}
If $G$ is a graph with $m$ edges and $\lambda (G)\ge \sqrt{m}$, then $G$ contains at least $\sqrt{m}$ triangular edges, unless $G$ is a complete bipartite graph. 
\end{conjecture}

If Conjecture~\ref{thm-main2} holds, then the bound $\sqrt{m}$ is best
possible. 
For positive integers $s,t\ge 2$, 
let $K_{s,t}^{+}$ denote the
graph obtained from $K_{s,t}$ by adding one
edge inside the part of size $s$. Taking $m$ as an integer such that $\sqrt{m}$ is odd, and setting  $s:=2(\sqrt{m}+1)$ and $t:=\tfrac{1}{2}(\sqrt{m}-1)$, one can verify that $K_{s,t}^+$ satisfies $\lambda (K_{s,t}^+) = \sqrt{m}$. 
The added edge of $K_{s,t}^+$  
together with the $t$ vertices in the opposite
part span a book of size $t$, which   
yields exactly $2t+1$ triangular edges. So $K_{s,t}^+$ contains exactly $2t+1=\sqrt{m}$ triangular edges. Thus, the conjectured bound is best possible. 

\medskip 
Theorem~\ref{thm-size-14} implies a weak form of
Conjecture~\ref{thm-main2}, losing only a factor of two. 

\begin{corollary}
Every $m$-edge Nosal graph contains more than $\tfrac{1}{2}\sqrt{m}$
triangular edges. 
\end{corollary}

For Theorem \ref{thm-size-14}, 
we in fact prove the following slightly stronger statement. 

\begin{theorem} \label{thm:book-free}
If $G$ is a $B_{r+1}$-free graph with $m$ edges, where $m\ge (4r)^2$, then
$$ \lambda(G)\le\sqrt{m}, $$
with equality if and only if $G$ is a complete bipartite
graph.
\end{theorem}

 Theorem~\ref{thm:book-free} fits into the Brualdi--Hoffman--Tur\'{a}n type
problems: for a forbidden subgraph $F$, one asks how large $\lambda(G)$
can be among all $F$-free graphs with $m$ edges, and which graphs attain the
maximum. 
This problem has received considerable
attention; see, e.g.,  
\cite{ZLS2021, LZS2024, LZZ2025jgt, LLLY2026aam, LLZ-edge-spectral}.

\subsection{Counting kites in Nosal graphs}

Beyond triangles, it is important to count other  substructures in Nosal graphs. 
 Ning and Zhai \cite{NZ2021b} studied the spectral supersaturation of $4$-cycles. 
Theorem~\ref{thm-size-14} shows that every Nosal graph contains not just a $C_3, C_4$ but a
copy of $B_2$. This makes $B_2$
the natural candidate to count. Let $C_4^+$ be the \emph{kite} obtained from $C_4$ by adding a chord (center edge); equivalently, $C_4^+$ is the book $B_2$. 
Here, copies of \(C_4^+\) are counted in the
usual non-induced sense. 
Since $C_4^+$ is color-critical (deleting the chord leaves a bipartite
$C_4$), a general result of Mubayi \cite{Mubayi2010} implies that 
every $n$-vertex graph with $m> \lfloor n^2/4\rfloor$ edges contains at least ${\lfloor n/2 \rfloor \choose 2}$ copies of $C_4^+$. However, 
this is an order-based count governed by the edge bound $\lfloor n^2/4\rfloor$.
The spectral condition $\lambda(G)>\sqrt{m}$ is of a different nature and is measured in terms of the number of edges $m$.   
Along this line, Li, Liu and Zhang~\cite{LiLiuZhang26} obtained the correct order of magnitude, proving that every Nosal graph contains $\Omega(m)$ copies of $C_4^+$. 
Recently, Fang, Lin and Zhai~\cite{FangLinZhai} studied edge-spectral
supersaturation results for color-critical graphs of chromatic number at least
$4$. The kite $C_4^+$ has chromatic number $3$, so this case falls outside the scope of their results, and this is the situation we treat here.

\smallskip 

As a direct application, Theorem~\ref{thm-size-14} already gives at least
$\big(\frac{1}{32}-o(1) \big)m$ copies of $C_4^+$ in every $m$-edge Nosal graph, since any two triangles in a largest
book form a kite. 
The fourth main result of this paper determines the 
optimal multiplicative constant $\tfrac18$.  

\begin{theorem} \label{thm-C4+}
Every $m$-edge Nosal graph $G$ contains at least 
$\left( \frac{1}{8} - o(1)\right) m$ copies of $C_4^+$. Moreover, the constant $\frac{1}{8}$ is best possible. In other words, we have 
\[  \lim\limits_{m \to \infty} \min_{\lambda (G)> \sqrt{m}} \frac{\texttt{\#}\,C_4^+(G)}{m} = \frac{1}{8}.  \]
\end{theorem}

\begin{remark}
We mention that although $C_4^+$ contains both $C_3$ and $C_4$ as subgraphs, 
the corresponding supersaturation phenomena differ substantially. 
Indeed, every Nosal graph contains at least about $\frac{1}{2}\sqrt{m}$ copies of $C_3$, 
$\frac{1}{8} m^2$ copies of $C_4$, and $\frac{1}{8}m$ copies of $C_4^+$. 
Thus, the number of $C_4^+$ lies between the numbers of its two subgraphs $C_3$ and $C_4$, so the extremal count of $C_4^+$ follows a pattern of its own, instead of being inherited  from that of $C_3$ or $C_4$. 
\end{remark}

\smallskip
For counting $C_3$ and $C_4$, there are two exact formulas: 
$$\texttt{\#}C_3(G)=\frac{1}{6} \sum_{i=1}^n \lambda_i^3$$ 
and (see \cite{LL2009DM}) 
$$\texttt{\#}C_4(G) = \frac{1}{8} \sum_{i=1}^n (\lambda_i^4 + \lambda_i^2) - \frac{1}{4} \sum_{i=1}^n d_i^2,$$
 where $\lambda_i$ and $d_i$ are eigenvalues and degrees of $G$, respectively. We refer to \cite{LNW2021, NingZhai2023, NZ2021b,LiLiuZhang26} for related  applications. 
A spectral Sidorenko-type bound for $K_{t,t}$ was recently established in \cite{LLLZ2026}. 
However, for the kite $C_4^+$ and the book $B_t\, (t\ge 3)$, no analogous formula is available for counting copies directly from the spectrum of $G$. 
This makes the problem of counting kites more challenging. 
Instead, we count the number of kites through the codegree formula: 
$$\texttt{\#}\,C_4^+(G)=\sum_{uv\in E(G)}\binom{\tau(uv)}{2}, \quad \text{where}~~\tau(uv):=|N(u)\cap N(v)|. $$ 

\paragraph{Proof overview.}
Both proofs of Theorems \ref{thm:book-free} and \ref{thm-C4+} share the same approach that is based on a local structural decomposition around a vertex of
maximum Perron weight, combining with an identity for the spectral surplus and a half-missing-weights estimate.   
Let $(x_v)_{v\in V(G)}$ be the Perron
eigenvector of $A(G)$ satisfying $\max_{v\in V(G)} \{x_v\}=1$, attained at a vertex
$u_*$. We denote 
\[ U:=N(u_*), \qquad W:=V(G)\setminus(\{u_*\}\cup U). \] 
To begin with, 
a standard step is to establish a formula that localizes the {spectral surplus}
$\lambda^2-m$. 
The Perron equation at $u_*$
gives $\lambda =\sum_{u\in U}x_u$. Multiplying both sides by $\lambda$ and substituting $\lambda x_u =\sum_{w\in N(u)} x_w$, the double counting then yields the following exact identity (Theorem \ref{thm:identity}) 
\begin{equation} \label{eq-dengshi}
\lambda^2-m=\sum_{uv\in E(U)}(x_u+x_v-1)-\Phi, 
\end{equation}
where $\Phi := \sum_{w\in W} \big(d_U(w)(1-x_w) + \frac{1}{2}d_W(w) \big) \ge 0$ is called a \emph{deficit term} supported on $W$. 
Thus, the spectral surplus $\lambda^2-m$ reflects the competition between the edges
inside $U$ and the deficit $\Phi$. 

Only \emph{bad} edges $uv\in E(U)$, those with $x_u+x_v >1$, 
contribute positively, and the remaining edges may be discarded. 
The missing-weight of $u\in U$ in $W$ is denoted by 
$$ Y(u):=\sum_{w\in W\setminus N(u)}x_w .$$ 
A key ingredient in our proofs is to establish a general estimate on missing-weights. In Theorem \ref{thm:halfmissing}, we shall show that $\Phi$ already controls
half of the missing-weights between $U$ and $W$:  
\begin{equation} \label{eq-half-miss}
     \frac12\sum_{u\in U}x_uY(u) \le \Phi. 
\end{equation} 
Therefore, the heart of the remaining argument is to link (\ref{eq-dengshi}) and (\ref{eq-half-miss}),  showing that the contribution of bad edges inside $U$ is dominated by half of the missing-weights:  
\begin{equation} \label{bad-edge-upper-bound}
  \sum_{uv\in E(U)}(x_u+x_v-1)\le \frac12\sum_{u\in U}x_uY(u). 
\end{equation}
The bound (\ref{bad-edge-upper-bound}) relies on the local structural properties of $G$. So the hypotheses of the two theorems enter solely through
bounds on the contribution of edges of $U$. 

\smallskip
\noindent
\boxed{\text{Sketch of Theorem~\ref{thm:book-free}}}
Being $B_{r+1}$-free forces a single constraint:  $d_U(u)\le r$ for all $u\in U$
(otherwise $u_*u$ would span a book of size $r+1$). Applying this together with the bound
$|N_W(u)\cap N_W(v)|\le r-1$ on the triangle $u_*uv$, we obtain the local 
bound $Y(u)\ge(\lambda x_v-2r)_+$ and 
$ Y(v)\ge (\lambda x_u -2r)_+$ for each bad edge $uv$, where  $(t)_+:=\max\{t,0\}$. 
An elementary inequality, 
$a+b-1\le\tfrac12\bigl(a(4b-2)_++b(4a-2)_+\bigr)$, together with the assumption $\lambda\ge 4r$,  yields the  bound
$$ x_u+x_v-1\le\frac1{2r}\bigl(x_uY(u)+x_vY(v)\bigr). $$  Summing over bad edges and
using $d_U (u)\le r$ again, we get 
$\sum_{\text{bad}} (x_u+x_v -1) \le  \tfrac12\sum_{u}x_uY(u)$, as desired in (\ref{bad-edge-upper-bound}). Combining with (\ref{eq-dengshi}) and (\ref{eq-half-miss}), 
we obtain $\lambda ^2\le m$. For the equality case, $\lambda^2=m$ first forces no bad edges in $U$ and then $\Phi=0$, so every vertex $w\in W$ has weight $1$ and
is adjacent to all vertices of $U$. A short computation with the Perron equations shows that any
edge inside $U$ would force $\lambda\le 3r<4r$, a contradiction, so $e(U)=0$ and
$G$ is complete bipartite.

\smallskip
\noindent
\boxed{\text{Sketch of Theorem~\ref{thm-C4+}}}
Sharpness comes from $G_t$, obtained from $K_{4t,t}$ by adding one
edge inside the part of size $4t$: the Rayleigh quotient shows that $G_t$ is Nosal, while all copies of $C_4^+$ contain the added edge as a chord, which yields 
$\texttt{\#}\,C_4^+(G_t)=\binom t2=(\tfrac18-o(1))e(G_t)$. For the lower bound, we argue by
contradiction, assuming that $\lambda^2>m$, but 
$\texttt{\#}\, C_4^+(G)\le(\tfrac18-\varepsilon)\lambda^2$.

 First, 
we sharpen the two local estimates to weighted forms
$Y(u)\ge \big(\lambda x_v-L_v(uv) \big)_+$, where  $L_v(uv):=1+\sum_{z\in N_U(v)} x_z + \sum_{w\in Q_{uv}} x_w$ with $Q_{uv}:=N_W(u)\cap N_W(v)$.   Second, we split the bad
edges of $U$ into \emph{ordinary} ones (that is, $L_u(uv)+L_v(uv)\le\lambda$) and
\emph{exceptional} ones. A weighted version of the elementary inequality in Lemma \ref{lem:kite-algebra}, combined with
a local coefficient estimate in Claim \ref{lem:kite-local-coeff}, shows that the ordinary edges are still absorbed
by the deficit $\Phi$:
\begin{equation} \label{bound-ordinary}  \sum_{\text{ordinary $uv$}}(x_u+x_v-1)\le \left( \frac{1}{2}-\delta_{\varepsilon}+o(1) \right) \sum_{u\in U} x_uY(u) \le \big( 1- 2\delta_{\varepsilon} + o(1) \big)\Phi, 
\end{equation} 
where the last inequality follows by (\ref{eq-half-miss}), and the gain
$\delta_{\varepsilon}$ is supplied by the contradiction hypothesis, which captures 
$d_U(u)\le(\tfrac12-\gamma_{\varepsilon} )\lambda$. 
If there is no exceptional edge, then
 (\ref{eq-dengshi}) and (\ref{bound-ordinary})  would give $\lambda^2-m\le-(2\delta_{\varepsilon} -o(1))\Phi\le0$, a contradiction.
Hence, an exceptional edge must exist. 
We write $\mathcal E$ for the set of exceptional bad edges. One verifies $|\mathcal E|=O(\lambda)$, which yields 
$\Phi=O(\lambda)$. Using (\ref{eq-half-miss}), the sum of missing-weights between $U$ and $W$ is $O(\lambda)$. A single
exceptional edge $uv$ then forces
\begin{equation} \label{eq-count-lower}
\texttt{\#} \,C_4^+(G)\ge 
\frac{1}{2} (A^2+B^2+C^2)+AC+BC-O(\lambda),
\end{equation}
where $A,B,C$ are the sums of weights of $N_U(u),N_U(v)$ and
$Q_{uv}$, respectively. 
The first three terms in (\ref{eq-count-lower}) come from the three center
edges $u_*u,u_*v,uv$, and the cross terms $AC,BC$ from two ``rectangles''
pairing $N_U(u),N_U(v)$ with $Q_{uv}$, all of whose center edges are distinct.
Minimizing the expression of (\ref{eq-count-lower}) under the exceptional constraint
$A+B+2C\ge\lambda-O(1)$ gives at least $\tfrac18\lambda^2-O(\lambda)$ copies of
$C_4^+$, contradicting the hypothesis. 
Therefore, we establish the sharp constant
$\tfrac18$. 

\smallskip 
\noindent
\boxed{\text{A more general result}} 
We will demonstrate in Section \ref{sec:large-books} that the above approach also works for counting books $B_t$ of arbitrary size $t \ge 3$, not just the book of size two. 
More generally, in Theorem \ref{thm:Bt-count}, we will prove the following stronger result for every $t\ge 1$,   
\begin{equation} \label{eq-heart} \tag{$\heartsuit$}
   \lim_{m\to\infty}\ \min_{\lambda(G)>\sqrt m}\
   \frac{\texttt{\#} B_t(G)}{m^{t/2}}\;=\;\frac{1}{t!\,2^{t}} .
\end{equation}
The two smallest cases are already known: when $t=1$, the count is simply the number of triangles, of order $\tfrac12\sqrt m$ by Ning and Zhai \cite{NingZhai2023},
and the case $t=2$ is Theorem~\ref{thm-C4+} with constant $\tfrac18$.

\paragraph{Organization.}
The rest of the paper is organized as follows. 
In Section~\ref{sec:2}, we prove Theorem~\ref{thm:m-2}. 
Section~\ref{sec:3} is devoted to the proofs of Theorems~\ref{thm-edge-spect-LS} and~\ref{thm:strong}, 
and Section~\ref{sec:4} to those of Theorems~\ref{thm:book-free}, \ref{thm-C4+} 
and (\ref{eq-heart}). 
We conclude in Section~\ref{sec:concluding} with an open problem for interested readers.

\section{The third-layer supersaturation}

\label{sec:2}
The proof of Theorem \ref{thm:m-2} separates into two cases according to
whether $G$ contains a copy of $K_4$.  In the first case,  we use a special case of
Nikiforov's inequality.  In the second case, interlacing gives
three eigenvalues at most $-1$, and a third-moment estimate forces enough
triangles.

First, the coefficient $\frac{1}{3}$ of Bollob\'{a}s--Nikiforov's bound (\ref{eq:BN}) can be improved to $\frac{1}{2}$ under a specific structural assumption. Let $T_i$ denote the number of all $i$-cliques in $G$.  
An inequality of Nikiforov \cite{Nikiforov2002} states that if $G$ is a $K_{k+1}$-free graph with spectral radius $\lambda$, then
\begin{equation*} 
\lambda^k \le T_2 \lambda^{k-2} + \cdots + (i-1)T_i \lambda^{k-i} + \cdots + (k-1)T_k.  
\end{equation*}

We shall use the following $k=3$ case of Nikiforov's inequality above.

\begin{lemma}[Nikiforov \cite{Nikiforov2002}]
\label{lem:K4free}
Let $G$ be an $m$-edge $K_4$-free graph with  spectral radius $\lambda $.  Then
\[
    t(G)\ge \frac{\lambda (\lambda^2-m)}{2}.
\]
\end{lemma}

We also need the following elementary inequality.

\begin{lemma}
\label{lem:cubes}
Let $L\ge 1$, and let $a_1,a_2,\dots,a_s$ be nonnegative real numbers such that at
least three of them are at least $1$ and
$\sum_{i=1}^s a_i^2\le L^2+2$. 
Then
\[
        \sum_{i=1}^s a_i^3\le L^3+2.
\]
Moreover, equality occurs only when, up to permutation, $ (a_1,a_2,\dots,a_s)=(L,1,1,0,\ldots ,0)$.
\end{lemma}

\begin{proof}
We denote $f(a_1,\ldots ,a_s)=\sum_{i=1}^s a_i^3$. 
    Without loss of generality, we may assume that 
    $a_1\ge \cdots \ge a_s\ge 0$ and $a_1,a_2,a_3\ge 1$. 
    First, it is easy to verify the following two inequalities: 
   \begin{equation} \label{eq-ele-1}
   (a^2+b^2)^{3/2} > (a^3+b^3) \qquad  \text{for all $a,b> 0$}; 
   \end{equation}
   and 
   \begin{equation} \label{eq-ele-2}
   (x^2+y^2-1)^{3/2} +1 > x^3+y^3 \qquad 
   \text{for all $x,y> 1$}.
   \end{equation}
   Starting from the tuple $(a_1,a_2,\ldots,a_s)$, we move each weight $a_j$, where $j\ge 4$ and $a_j >0$, to $a_1$ by replacing $(a_1,a_j)$ with $\big((a_1^2+a_j^2)^{1/2},0 \big)$. Using (\ref{eq-ele-1}), we see that this replacement 
   does not decrease the function $f$, and does not change the sum of squares of the weights. 
   After the operation for all $j>3$, we are left with three nonzero weights at least $1$.

For $i=2,3$, if $a_i>1$, then replacing $(a_1,a_i)$ with 
$\big( (a_1^2+a_i^2 -1)^{1/2},1 \big)$. Again, the sum of squares of weights is unchanged, and (\ref{eq-ele-2}) implies that the function $f$ does not decrease. After these operations, the resulting tuple has the form $(a_1^*,1,1,0,\ldots ,0)$, where $(a_1^*)^2+2 =\sum_{i=1}^s a_i^2 \le L^2+2$, which yields $a_1^* \le L$. Thus, we conclude that $f(a_1,a_2,\ldots ,a_s)\le f(a_1^*,1,1,0,\ldots ,0)\le L^3+2$, as needed.  
For the equality case, since the operations using inequality (\ref{eq-ele-1}) or (\ref{eq-ele-2}) strictly increase the value of $f$ whenever they change the tuple, the equality $\sum_{i=1}^s a_i^3 =L^3+2$ forces no such operation to be performed. Hence, the tuple is already of the form $(L,1,1,0,\ldots ,0)$. 
\end{proof}

We also need the following standard spectral characterization.

\begin{lemma} \label{lem-one-positive}
If a connected graph $H$ has exactly one positive  eigenvalue, then $H$ is complete multipartite. Moreover, a complete multipartite graph with $s\ge 2$ parts has rank $s$. 
\end{lemma}

\begin{proof}
Let $A=(a_{ij})$ be the adjacency matrix of $H$, let $\lambda$ be its spectral radius, and let $\bm{z}>0$ be the positive Perron eigenvector. Since $\lambda$ is the only positive eigenvalue, 
$A$ is negative semidefinite on the subspace $\bm{z}^\perp$. 
If $u$ and $v$ are non-adjacent, by denoting $\bm{y}=z_v \bm{e}_u-z_u \bm{e}_v$,  
then $\bm{y} \perp \bm{z}$ and $\bm{y}^T A \bm{y}=0$, so $A\bm{y}=0$. Hence, for every vertex $w$, we have $z_v a_{wu}=z_u a_{wv}$. 
As \(z_u,z_v>0\), we have $a_{wu}=a_{wv}\in\{0,1\}$, so 
 $u$ and $v$ have the same neighborhood. 
Thus, non-adjacency relation is transitive: 
its equivalence classes are independent sets, and any two distinct classes are joined completely. Hence, \(H\) is a
complete multipartite graph on the equivalence classes.

For the rank assertion, vertices in the same part give identical columns, so the rank is at
most $s$. Taking one representative from each part gives the principal submatrix $J_s -I_s$, which is nonsingular for $s\ge 2$. Hence the rank of $A(H)$ is exactly $s$.
\end{proof} 

\begin{lemma}\label{lem:rank-two-bipartite}
If a graph \(G\) has exactly two nonzero adjacency eigenvalues, one positive
and one negative, then the non-isolated part of \(G\) is a complete bipartite
graph.
\end{lemma}

\begin{proof}
The adjacency matrix has rank two and trace zero, so its two nonzero eigenvalues
are \(\lambda\) and \(-\lambda\). In particular \(\operatorname{tr}A^3=0\), so
\(G\) is triangle-free. Also \(G\) has exactly one positive eigenvalue, hence
its unique nontrivial component is complete multipartite by
Lemma~\ref{lem-one-positive}. A complete multipartite triangle-free graph has
at most two parts, and hence is complete bipartite.
\end{proof}

We are now ready to prove Theorem \ref{thm:m-2}. 

\begin{proof}[{\bf Proof of Theorem \ref{thm:m-2}}]
We denote $ \lambda=\lambda(G)$ and $t=t(G)$. 
The assumption is equivalent to 
\begin{equation} \label{eq:assumption}
        m\le \lambda^2-2\lambda+3.
\end{equation} 
We first consider the case where $G$ is $K_4$-free.  Applying Lemma
\ref{lem:K4free}, we have 
\[
        t \ge \frac{\lambda(\lambda^2-m)}{2}.
\]
Using (\ref{eq:assumption}), we get
$\lambda^2-m\ge 2\lambda-3$. 
Hence, we get 
\[
        t\ge \frac{\lambda(2\lambda-3)}{2} 
        \ge (\lambda -1)^2,
\]
where the last inequality holds since $\lambda \ge 1+ \sqrt{m-2} \ge 2$. 
Then we have $ t\ge (\lambda-1)^2\ge m-2$. 

Now suppose that $G$ contains a copy of $K_4$.  Let $ \lambda_1=\lambda\ge \lambda_2\ge \cdots\ge \lambda_n$ be the eigenvalues of the
adjacency matrix $A$ of $G$. 
The eigenvalues of $K_4$ are $3,-1,-1,-1$.  By Cauchy interlacing, applied to
the principal submatrix induced by a copy of $K_4$, we have
\begin{equation} \label{eq:three}
        \lambda_{n-2},\ \lambda_{n-1},\ \lambda_n\le -1.
\end{equation}
We denote  
\[
        L:=\lambda-2.
\]
Since $G$ contains $K_4$, we have $m\ge 6$, and so the hypothesis gives
$\lambda\ge 3$.  Thus $L\ge 1$.

Using $\sum_{i=1}^n \lambda_i^2=\tr(A^2)=2m$ and (\ref{eq:assumption}), we obtain
\begin{equation} \label{eq-square}
\begin{aligned}
        \sum_{i=2}^n\lambda_i^2
        =2m-\lambda^2                           \le 2(\lambda^2-2\lambda+3)-\lambda^2   =(\lambda-2)^2+2                        =L^2+2.
\end{aligned}
\end{equation}
Among the values $\lambda_2,\lambda_3,\ldots ,\lambda_n$, 
we write the negative eigenvalues as 
$ -a_1,\dots,-a_s$, where each $a_i>0$, 
and the positive eigenvalues as 
$ b_1,\dots,b_r$, where each $b_j>0$. 
By (\ref{eq:three}), at least three of the $a_i$ are at least $1$.  From (\ref{eq-square}), we see that  
\[
        \sum_i a_i^2+\sum_j b_j^2\le L^2+2.
\]
Applying Lemma \ref{lem:cubes} gives
$  \sum_{i} a_i^3\le L^3+2$. 
Therefore, we get 
\[
        \sum_{i=2}^n\lambda_i^3 
        =
        \sum_j b_j^3-\sum_i a_i^3
        \ge -\sum_i a_i^3 \ge 
        -(L^3+2).
\]
Since $\tr(A^3)=6t$, we get
\[
\begin{aligned}
        6t
        =\sum_{i=1}^n\lambda_i^3                \ge \lambda^3-(\lambda-2)^3-2       =6(\lambda-1)^2.
\end{aligned}
\]
By the hypothesis $\lambda \ge 1+ \sqrt{m-2}$, we conclude $ t\ge m-2$, as desired. 

\smallskip 
It remains to determine the equality case.  

Suppose first that $G$ is
$K_4$-free and $t=m-2$.  In the chain of inequalities above, equality must hold
throughout.  In particular,
we must have $   \frac{\lambda(2\lambda-3)}{2}=(\lambda-1)^2$, 
so $\lambda=2$.  Since
$  2=\lambda\ge 1+\sqrt{m-2}$, 
we get $m\le 3$.  As $m\ge 3$, it follows that $m=3$, $t=1$ and $G=K_3$.

Now suppose that $G$ contains $K_4$ and $t=m-2$.  From the proof, the equality forces
\[
        (\lambda-1)^2=m-2. 
\]
It also forces equality in the third-moment estimate.  In the chain
\[
        \sum_{i=2}^n\lambda_i^3
        =
        \sum_j b_j^3-\sum_i a_i^3
        \ge
        -\sum_i a_i^3
        \ge
        -(L^3+2),
\]
both inequalities must therefore be equalities.  Hence $ \sum_j b_j^3=0$, 
so there are no positive non-principal eigenvalues.  Since $ \sum_i a_i^3=L^3+2$,  
 the equality statement in Lemma \ref{lem:cubes} gives, up to order,  
$ a_1=L$ and $a_2=a_3=1$, 
and there are no further negative non-principal eigenvalues. Thus, all
remaining non-principal eigenvalues are $0$.  Then the nonzero eigenvalues
of $G$ are $-(\lambda-2), -1,-1,\lambda$.  
In particular, $G$ has exactly one positive adjacency eigenvalue.

By Lemma \ref{lem-one-positive}, 
$G$ is a complete multipartite graph.  Let its part sizes be $  s_1,\dots,s_p \ge 1$. 
Since $G$ contains $K_4$, we have $p\ge 4$. 
If $p\ge 5$, then every edge lies in at least $p-2\ge 3$ triangles. Indeed, an
edge joins two different parts, and a third vertex may be chosen from any of
the other $p-2$ nonempty parts.  Hence
by double counting, we get $3t(G)\ge m(p-2)\ge 3m$, so $t\ge m$, contradicting $t=m-2$.  Therefore, we must have $p=4$, so $G$ is a complete four-partite graph
with part sizes $s_1,s_2,s_3,s_4$.  Its number of edges and triangles are
\[
        m=\sum_{1\le i<j\le 4}s_is_j,
        \qquad
        t=\sum_{1\le i<j<\ell\le 4}s_is_js_\ell.
\]
We denote $ s_i:=1+x_i$ for some $ x_i\ge 0$. 
A direct expansion gives
\begin{equation} \label{eq-zero}
\begin{aligned}
        t-m+2
        &=
        \sum_{1\le i<j\le 4}x_ix_j
        +
        \sum_{1\le i<j<\ell\le 4}x_ix_jx_\ell .
\end{aligned}
\end{equation}
Since $t=m-2$, the left-hand side of (\ref{eq-zero}) is $0$.  The right-hand side is a sum
of nonnegative integers, so every term is $0$.  Hence at most one of the
$x_i$ is nonzero.  Therefore the four part sizes are $s_1,1,1,1$ 
for some integer $s_1\ge 1$.  Equivalently, we get $G= K_3\vee s_1K_1$, where $s_1=\frac{m-3}{3}$. 
\end{proof}

\section{The edge-spectral Lov\'{a}sz--Simonovits theorem}

\label{sec:3}

In this section, we first prove Theorem \ref{thm:strong}, and then apply it to prove Theorem \ref{thm-edge-spect-LS}.  
To prove Theorem \ref{thm:strong}, we need to use the following power-mean monotonicity of Jensen; see, e.g., \cite{Beck1946}. 

\begin{lemma} \label{lem:Jensen}
Suppose that $a_1,\ldots ,a_n$ are nonnegative numbers. Then for $s\ge t>0$,  
\[ \left(\sum_{i=1}^n a_i^s \right)^{\!1/s} \le 
\left(\sum_{i=1}^n a_i^t \right)^{\!1/t} .  \] 
\end{lemma}

The following elementary inequality is also needed for our purpose.

\begin{lemma}\label{lem:scalar}
For every real number $r$ satisfying
$ 1\leq r< \sqrt 2$, 
one has
\[
    r^3-(2-r^2)^{3/2}\geq 6(r-1),
\]
with equality if and only if $r=1$. 
\end{lemma}

\begin{proof}
Let $h(r)=r^3-(2-r^2)^{3/2}-6(r-1)$. Then $h(1)=0$, and
\[
    h'(r)=3r^2+3r\sqrt{2-r^2}-6=3\bigl(r^2+r\sqrt{2-r^2}-2\bigr).
\]
For $r\in[1,\sqrt2 \, )$, we have $r\ge1 \ge\sqrt{2-r^2}$, hence
$r\sqrt{2-r^2}\ge 2-r^2$ and therefore $r^2+r\sqrt{2-r^2}\ge 2$, so
$h'(r)\ge 0$ for every $r\in [1,\sqrt2 \, )$. Thus, $h(r)$ is increasing and $h(r)\ge h(1)=0$.

For the equality case, note that $r\sqrt{2-r^2}= 2-r^2$ forces $r=\sqrt{2-r^2}$, i.e., $r=1$. So $h'(r)=0$ only at
$r=1$; in particular $h'(r)>0$ and $h$ is
strictly increasing on $[1,\sqrt2\,)$, and therefore $h(r)>h(1)=0$ for every
$r\in(1,\sqrt2\,)$. Thus, the equality holds if and only if $r=1$. 
\end{proof}

Now, we are ready to prove Theorem \ref{thm:strong}. 

\begin{proof}[{\bf Proof of Theorem~\ref{thm:strong}}] 
We denote $ \lambda=\lambda(G)$.  
If $\lambda < \sqrt m$, there is nothing to show. We now assume that $\lambda \geq \sqrt m$. 
Applying Lemma~\ref{lem:Jensen}, we have 
\[ -\left( \sum_{i=2}^n\lambda_i^3 \right)^{1/3} \le \left( \sum_{i=2}^n |\lambda_i|^3 \right)^{1/3} \le \left( \sum_{i=2}^n\lambda_i^2 \right)^{1/2}, \]
which gives 
\[
    \sum_{i=2}^n\lambda_i^3
    \geq
    -\left(\sum_{i=2}^n\lambda_i^2\right)^{3/2}.
\]
Combining with $ 6t(G)= \mathrm{tr}\,A(G)^3 = \lambda^3+\sum_{i=2}^n\lambda_i^3$, 
we get 
\[
    6t(G)
    \geq
    \lambda^3-
    \left(\sum_{i=2}^n\lambda_i^2\right)^{3/2}.
\]
Since $  \sum_{i=2}^n\lambda_i^2=2m-\lambda^2$, the above inequality can be written as  
\begin{equation}\label{eq:spectral-lower}
    6t(G)
    \geq
    \lambda^3-(2m-\lambda^2)^{3/2}.
\end{equation}
We denote $ \lambda:= r{\sqrt m}$ for some $r >0$.  
The hypothesis of Theorem \ref{thm:strong} gives $r\ge 1$. Moreover, it is well-known that $\lambda <\sqrt{2m}$, which yields $r< \sqrt{2}$. Thus, 
we conclude that $$ 1\le r< \sqrt{2}. $$ 
Note that the right-hand side of \eqref{eq:spectral-lower} 
is 
\[
    \lambda^3-(2m-\lambda^2)^{3/2}
    =m^{3/2}\left(r^3-(2-r^2)^{3/2}\right).
\]
Applying Lemma~\ref{lem:scalar} gives 
$ r^3-(2-r^2)^{3/2}\geq 6(r-1)$.  
Hence
\begin{equation*} 
      \lambda^3-(2m-\lambda^2)^{3/2}
    \geq
    m^{3/2} \cdot 6(r-1) = 6m \big(\lambda - \sqrt{m}\, \big).
\end{equation*}
Combining with (\ref{eq:spectral-lower}), we obtain that 
$ t(G)\geq m(\lambda-\sqrt m\,)$, as needed. 

It remains to characterize the equality. If $t(G)=m(\lambda-\sqrt m\,)$, then
 $-\sum_{i=2}^n\lambda_i^3 = \sum_{i=2}^n |\lambda_i|^3=(\sum_{i=2}^n\lambda_i^2)^{3/2}$, which forces at most one of $\lambda_2,\dots,\lambda_n$ to be nonzero, and that nonzero eigenvalue is negative. 
Thus, $A(G)$ has at most two nonzero eigenvalues, the one is positive and the other is negative. By Lemma \ref{lem:rank-two-bipartite}, $G$ is a complete bipartite
graph. In this case, $\lambda=\sqrt m$ and $t(G)=0$. 
\end{proof}

Next, we are in a position to prove Theorem \ref{thm-edge-spect-LS}. 

\begin{proof}[{\bf Proof of Theorem \ref{thm-edge-spect-LS}}]
    Assume that $G$ is an $m$-edge graph with $\lambda (G)\ge \sqrt{m} +q$ for some real number $q>0$. Then $G$ is not bipartite since $\lambda (G)> \sqrt{m}$. 
    Applying Theorem \ref{thm:strong}, we get 
    $$ t (G)> m \big( \lambda (G)- \sqrt{m}\, \big)\ge qm. $$
    Next, we show that 
the above bound is tight. The split graph $S_{k,m}$ is defined as follows: For each $1\le k<m$, let $r,t$ be integers such that $m- {k \choose 2}= k t + r$, where $0\le r\le k-1$. 
Let $S_{k,m}$ be an $m$-edge graph 
obtained from the join graph $K_{k}\vee tK_1$ by adding an extra vertex that has exactly $r$ neighbors in the vertex set of $K_k$. By computation (see, e.g., \cite[Lemma 3.3]{LLLZ2026}), we know that 
$$ \lambda (S_{k,m})=\sqrt{m} + \frac{k-1}{2} + O_k(m^{-1/2}). $$
Clearly, we have  $t(S_{k,m})= {k \choose 3} + t {k \choose 2} + {r \choose 2} = \frac{k-1}{2}m-O_k(1)$.  
Comparing these with Theorem \ref{thm-edge-spect-LS}, the split graph  construction corresponds to $q\approx \frac{k-1}{2}$, while $t(S_{k,m})\approx \frac{k-1}{2}m = q\cdot m$. The two sides are asymptotically equal, 
so the coefficient $q$ in the bound $t(G)$ cannot be improved. 
\end{proof}

\section{Books and kites in Nosal graphs}
\label{sec:4}

In this section, we prove Theorem~\ref{thm:book-free} and Theorem~\ref{thm-C4+}.
We begin with two theorems that will be used in both proofs. 
Let $G$ be a graph with $m$ edges, and let $(x_v)_{v\in V(G)}$ be the Perron eigenvector for $A(G)$, normalized so that $\| \bm{x}\|_{\infty} = \max_{v\in V(G)}x_v=1$. 
Fix a vertex $u_*$ with $x_{u_*}=1$ and denote $U:=N(u_*)$ and $W:=V(G)\setminus (\{u_*\}\cup U)$.  
For a vertex $v$ and a set $U$, we write $d_U(v):=|N(v)\cap U|$.   
We will frequently use the fact that if $G$ is $B_{r+1}$-free, then  $d_U(u)\le r$ for every $u\in U$. Otherwise, if $d_U(u)\ge r+1$, then the edge $u_*u$ spans a book of size $r+1$, a contradiction. 
The starting point is to estimate the contribution of edges of $U$ to the spectral surplus $\lambda^2(G)-m$.

\begin{theorem} \label{thm:identity}
With the notation above, we have 
\[
\lambda^2(G)-m=\sum_{uv\in E(U)}(x_u+x_v-1)-\Phi,
\]
where
\[
\Phi:=\sum_{w\in W}\Bigl(d_U(w)(1-x_w)+\frac12 d_W(w)\Bigr) \ge 0.
\]
\end{theorem}

\begin{proof}
Since $x_{u_*}=1$, the Perron equation at $u_*$ gives
\[
\lambda(G)=\sum_{u\in U}x_u.
\]
Now sum the Perron equations over all vertices of $U$. For each $u\in U$,
\[
\lambda(G)x_u=1+\sum_{v\in N_U(u)}x_v+\sum_{w\in N_W(u)}x_w,
\]
because each $u\in U$ is adjacent to $u_*$, its neighbors inside $U$ are exactly $N_U(u)$, and its neighbors inside $W$ are exactly $N_W(u)$.
Summing over $u\in U$ gives
\[ \lambda^2 (G) = 
\lambda(G)\sum_{u\in U}x_u
=
|U|+\sum_{u\in U}\sum_{v\in N_U(u)}x_v+\sum_{u\in U}\sum_{w\in N_W(u)}x_w.
\]
By the double counting, in the second term, the weight $x_v$ appears exactly $d_U(v)$ times, and in the third term, the weight $x_w$ appears exactly $d_U(w)$ times. 
Hence
\[
\lambda^2 (G)=|U|+\sum_{v\in U}d_U(v)x_v+\sum_{w\in W}d_U(w)x_w.
\] 
Moreover, we observe that 
\[
m=|U|+e(U)+e(U,W)+e(W).
\] 
On the other hand, we have
\[
\sum_{uv\in E(U)}(x_u+x_v-1)=\sum_{v\in U}d_U(v) x_v -e(U),
\]
and 
\[ \Phi =e(U,W) - \sum_{w\in W} d_U(w)x_w + e(W). \]
From the above four identities, it is easy to see that 
\[
\lambda^2(G)-m=\sum_{uv\in E(U)}(x_u+x_v-1)- \Phi,
\]
which is the required identity.
\end{proof}

We now show that $\Phi$ controls a half amount of missing weights of vertices of $U$.

\begin{theorem} \label{thm:halfmissing}
For each $u\in U$, we define the missing-weights of $u$ as  
$$ Y(u):=\sum_{w\in W\setminus N(u)}x_w. $$
Then  
\[
\Phi\ge \frac12\sum_{u\in U}x_uY(u).
\]
\end{theorem}

\begin{proof}
Fix a vertex $w\in W$, we write
\[ X(w):=\sum_{u\in U\setminus N(w)}x_u,\qquad 
S(w):=\sum_{z\in N_W(w)}x_z.
\]
Since $\sum_{u\in U}x_u=\lambda(G)$, the Perron equation at $w$ becomes
\[
\lambda(G)x_w=\sum_{u\in N_U(w)}x_u+S(w)=\bigl(\lambda(G)-X(w)\bigr)+S(w).
\]
Hence
\[
S(w)=X(w)-\lambda(G)(1-x_w).
\]
Because every Perron coordinate is at most $1$, we have
\[
d_U(w)\ge \sum_{u\in N_U(w)}x_u=\lambda(G)-X(w),
\qquad
 d_W(w)\ge S(w).
\]
Therefore
\begin{align*}
d_U(w)(1-x_w)+\frac12 d_W(w) 
&\ge (\lambda(G)-X(w))(1-x_w)+\frac12 S(w)\\
&=(\lambda(G)-X(w))(1-x_w)+\frac12\bigl(X(w)-\lambda(G)(1-x_w)\bigr)\\
&=\frac12 x_wX(w)+\frac12(1-x_w)\bigl(\lambda(G)-X(w)\bigr)\\
&\ge \frac12 x_wX(w),
\end{align*}
where the last inequality holds since $X(w)\le \sum_{u\in U}x_u=\lambda(G)$.

Now summing over all $w\in W$, we obtain 
\[
\Phi \ge \frac12\sum_{w\in W}x_wX(w).
\]
Expanding the last sum and switching the order of summation,
\begin{align*}
\sum_{w\in W}x_wX(w)
=\sum_{w\in W}x_w\sum_{u\in U\setminus N(w)}x_u 
=\sum_{u\in U}x_u\sum_{w\in W\setminus N(u)}x_w =\sum_{u\in U}x_uY(u).
\end{align*} 
Thus, we conclude that 
$\Phi\ge \frac12\sum_{u\in U}x_uY(u)$, 
as claimed.
\end{proof}

\begin{remark}
    The above Theorem~\ref{thm:halfmissing}, as well as the previous Theorem~\ref{thm:identity}, do not use the $B_{r+1}$-freeness of $G$. Both of them rely
only on the Perron equation and the normalization $\|\bm{x}\|_{\infty} =1$.  
\end{remark}

\subsection{Proof of Theorem \ref{thm:book-free}}

We are going to prove Theorem \ref{thm:book-free}. 
An edge $uv\in E(U)$ is called \emph{bad} if
$x_u+x_v>1$. Let $\mathcal{B}$ be the set of all bad edges. 
Only bad edges of $U$ contribute positively to the identity in Theorem~\ref{thm:identity}.

\begin{lemma} \label{lem:Yu-lower}
Let $G$ be a $B_{r+1}$-free graph, and let $uv\in E(U)$ be a bad edge. Then
\[
Y(u)\ge (\lambda(G)x_v-2r)_+
\quad \text{and} \quad 
Y(v)\ge (\lambda(G)x_u-2r)_+,
\]
where $(t)_+:=\max\{t,0\}$.
\end{lemma}

\begin{proof}
We prove the first inequality; the second follows by symmetry.

Apply the Perron equation at the vertex $v$. Since $v\in U$, $u_*v\in E(G)$ and $x_{u_*}=1$, we have
\[
\lambda(G)x_v=1+\sum_{z\in N_U(v)}x_z+\sum_{w\in N_W(v)}x_w.
\]
Since $G$ is $B_{r+1}$-free, we have $\sum_{z\in N_U(v)}x_z\le d_U(v)\le r$. 
Then 
\[
\sum_{w\in N_W(v)}x_w\ge \lambda(G) x_v -1-r.
\]
Since the edge $uv$ already lies in the triangle $u_*uv$, and $G$ is $B_{r+1}$-free, we get $|N_W(u)\cap N_W(v)|\le r-1$. 
Every Perron coordinate is at most $1$, so 
\begin{align*}
\sum_{w\in N_W(v)\setminus N_W(u)}x_w
\ge \sum_{w\in N_W(v)}x_w-(r-1) 
\ge \lambda(G)x_v -2r.
\end{align*}
But $N_W(v)\setminus N_W(u)\subseteq W\setminus N(u)$, so
\[
Y(u)=\sum_{w\in W\setminus N(u)}x_w\ge \lambda(G) x_v-2r.
\]
Since $Y(u)\ge 0$ automatically, this implies
$Y(u)\ge (\lambda(G)x_v-2r)_+$. 
\end{proof}

The next lemma is an elementary inequality. 
This lemma indicates that the assumption $m\ge (4r)^2$ in Theorem \ref{thm:book-free} is the natural limit of this particular estimate. 

\begin{lemma} \label{lem:algebra}
For all $a,b\in[0,1]$ with $a+b>1$, we have 
\[
a+b-1
\le
\frac12\Bigl(a(4b-2)_+ + b(4a-2)_+\Bigr).
\]
\end{lemma}

\begin{proof}
Since $a+b>1$ excludes the case $a<\tfrac12$ and $b<\tfrac12$,  
we distinguish three cases.

\smallskip
\noindent\textbf{Case 1:} $a\ge \tfrac12$ and $b\ge \tfrac12$.
Then both positive parts are active, so
\[
\frac12\Bigl(a(4b-2)+b(4a-2)\Bigr)=4ab-a-b.
\]
The desired inequality follows by $(4ab-a-b)-(a+b-1)=(2a-1)(2b-1)\ge 0$.

\smallskip
\noindent\textbf{Case 2:} $a\ge \tfrac12$ and $b<\tfrac12$.
Then $(4b-2)_+=0$, so the right-hand side becomes
\[
\frac12\,b(4a-2)=b(2a-1).
\]
Again the inequality holds since $b(2a-1)-(a+b-1)=(1-a)(1-2b)\ge 0$.

\smallskip
\noindent\textbf{Case 3:} $a<\tfrac12$ and $b\ge \tfrac12$.
This is symmetric to Case 2.
\end{proof}

Now, we are ready to prove Theorem~\ref{thm:book-free}, and list the key steps in Figure \ref{fig:book-proof}. 

\begin{figure}[htbp]
\centering
\begin{tikzpicture}[
    node distance=5mm and 12mm,
    box/.style={rectangle, rounded corners, draw, align=center,
                text width=66mm, inner sep=3pt, font=\footnotesize},
    arr/.style={-{Stealth[length=1.8mm]}, semithick}
  ]

  \node[box] (setup)
        {$G$ is $B_{r+1}$-free, $m\ge(4r)^2$; WLOG, $\lambda:=\lambda(G)\ge 4r$.
         Take a vertex $u_*$ with $x_{u_*}=1$, set $U=N(u_*)$ and $W=V\setminus(\{u_*\}\cup U)$};
  \node[box, below=of setup] (tools)
        {Two general tools (Theorems~\ref{thm:identity} and \ref{thm:halfmissing}):
         Surplus identity: $\! \lambda^2-m=\sum_{uv}(x_u+x_v-1)-\Phi$
         \newline Half missing-weights: $\Phi\ge\frac12\sum_{u\in U}x_uY(u)$};

  \node[box, anchor=north, at={($(setup.north)+(78mm,0)$)}] (alg) {Using the inequality (Lemma~\ref{lem:algebra}) and local bound (Lemma~\ref{lem:Yu-lower}): for each bad edge $uv$, \\  
         $x_u+x_v-1\le\frac1{2r}\big(x_uY(u)+x_vY(v)\big)$};
         
  \node[box, below=of alg] (sum)
        {Summing over $\mathcal B$ and 
        using $ d_U(u)\le r$ gives 
         $\sum_{uv}(x_u+x_v-1)\le\frac12\sum_{u\in U}x_uY(u)\le\Phi$. \\
        The surplus identity then gives $\lambda\le\sqrt m$}; 

  \draw[arr] (setup)--(tools);

  \draw[arr] (alg)--(sum);

  \draw[arr] (tools.east) -| ($(tools.east)!0.5!(alg.west)$) |- (alg.west);
\end{tikzpicture}
\caption{The flowchart of the proof of Theorem~\ref{thm:book-free}.}
\label{fig:book-proof}
\end{figure}
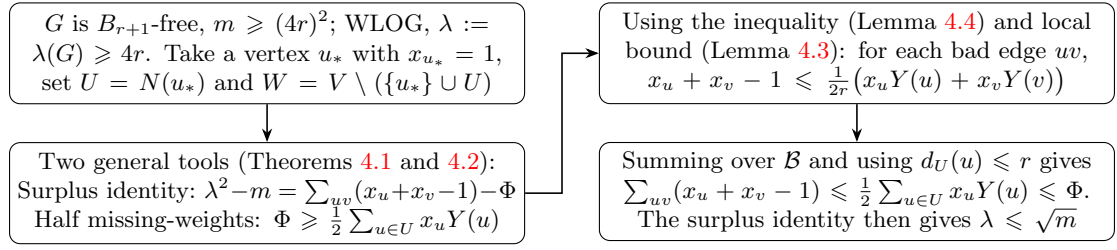

\begin{proof}[{\bf Proof of Theorem~\ref{thm:book-free}}]
Assume that $m\ge (4r)^2$ and $G$ is a $B_{r+1}$-free graph with $m$ edges. Throughout the proof, we may assume that $G$ has no isolated vertices. 
We may assume that $\lambda (G)\ge 4r$, since otherwise $\lambda(G)< 4r\le \sqrt{m}$, we are done. 
Our goal is to prove that $\lambda(G)\le \sqrt{m}$, with equality if and only if $G$ is a complete bipartite graph. We denote $\lambda :=\lambda (G)$. 

If $r=0$, then $G$ is triangle-free, so $E(U)=\emptyset$. 
Theorem~\ref{thm:identity} then yields 
\[ \lambda^2-m=-\Phi\le 0, \] 
so $\lambda \le \sqrt{m}$, with equality if and only if $\Phi =0$. 
If $W=\emptyset$, then we are done. 
If $W\neq \emptyset$, then $\Phi =0$ forces $e(W)=0$ and $d_U(w)(1-x_w)=0$ for each $w\in W$. The latter gives $d_U(w)=0$ or $x_w=1$.
If $d_U(w)=0$, then $w$ is isolated, a contradiction. 
Hence, $x_w=1$ and the Perron equation at $w$ forces $N(w)=U$.  
Thus, $G$ is a complete bipartite graph with parts $U$ and $\{u_*\}\cup W$.  

Now we assume that $r\ge 1$. 
Recall that $\mathcal{B}$ denotes the set of bad edges of $U$, that is,
\[
\mathcal{B}:=\{uv\in E(U):x_u+x_v>1\}.
\]
For each bad edge $uv\in\mathcal{B}$, setting $a:=x_u$ and $b:=x_v$ in Lemma~\ref{lem:algebra}, we get 
\begin{equation} \label{eq-after-4-4}
x_u + x_v-1 
\le
\frac12\Bigl(x_u(4x_v-2)_+ + x_v(4x_u-2)_+\Bigr) 
\le \frac{1}{2r} \Big( x_u(\lambda x_v-2r)_+ + x_v(\lambda x_u-2r)_+\Big),
\end{equation}
where the second inequality holds since $\lambda \ge 4r$.

Applying Lemma~\ref{lem:Yu-lower} gives 
$(\lambda x_v-2r)_+\le Y(u)$ and $(\lambda x_u -2r)_+\le Y(v)$. 
Hence, for each bad edge $uv$, we conclude that 
\[
x_u+x_v-1\le \frac{1}{2r}\Big( x_uY(u)+x_vY(v) \Big).
\]
Summing this over all bad edges gives
\begin{align*}
\sum_{uv\in \mathcal{B}}(x_u+x_v-1)
\le \frac1{2r}\sum_{uv\in \mathcal{B}}\bigl(x_uY(u)+x_vY(v)\bigr) 
=\frac1{2r}\sum_{u\in U}x_uY(u)\,d_{\mathcal{B}}(u),
\end{align*}
where $d_{\mathcal{B}}(u)$ is the number of bad edges of $\mathcal{B}$ incident with $u$. 

Since $G$ is $B_{r+1}$-free, we get $d_{\mathcal{B}}(u)\le d_U(u)\le r$, so 
\[
\sum_{uv\in \mathcal{B}}(x_u+x_v-1)
\le
\frac12\sum_{u\in U}x_uY(u) \le \Phi,
\]
where the last inequality holds by  Theorem~\ref{thm:halfmissing}.

Note that every edge of $E(U)\setminus\mathcal{B}$ satisfies $x_u+x_v-1\le 0$. We conclude that 
\[
\sum_{uv\in E(U)}(x_u+x_v-1)
\le
\sum_{uv\in \mathcal{B}}(x_u+x_v-1)
\le \Phi.
\]
Finally, applying Theorem~\ref{thm:identity} yields 
$
\lambda^2-m 
\le 0$. 
Thus, we get $\lambda \le \sqrt{m}$.

\smallskip 
\noindent 
{\bf Equality case.}
It remains to characterize the equality case. Suppose that $\lambda^2=m$. First, 
we claim that there are no bad edges in $U$. Suppose for a contradiction that $uv\in\mathcal B$. Since $x_u+ x_v>1$, one of $x_u,x_v$ is larger than $1/2$. By symmetry, we may assume $x_v>1/2$. Then
\[
\lambda(G)x_v-2r\ge 4rx_v-2r>0.
\]
In the proof of Lemma~\ref{lem:Yu-lower}, the estimate
$Y(u)\ge \lambda(G)x_v-2r$ 
used the inequality
\[
\sum_{z\in N_U(v)}x_z\le d_U(v)\le r.
\]
Tightness of the local estimate for the bad edge $uv$ therefore forces
\[
\sum_{z\in N_U(v)}x_z=r.
\]
Since $d_U(v)\le r$ and every Perron coordinate is at most $1$, this implies that every neighbor of $v$ inside $U$ has Perron coordinate $1$. In particular, $u\in N_U(v)$ gives $x_u=1$. 
But then the inequality in Lemma~\ref{lem:algebra} is strict for the pair $(x_u,x_v)$, since $x_u=1$ and $x_v>1/2$. 
This contradicts the required tightness of the local inequality for the bad edge $uv$. Hence
$\mathcal B= \emptyset$. 

Second, we will derive a contradiction with $\lambda \ge 4r$. 
Since there are no bad edges, every edge $uv\in E(U)$ satisfies $x_u+x_v-1\le 0$. Equality in Theorem~\ref{thm:identity} and $\Phi\ge 0$ now imply
\[
\Phi=0
\qquad\text{and}\qquad
x_u+x_v=1\quad\text{for every }uv\in E(U).
\]
As discussed above, the condition $\Phi=0$ implies that $e(W)=0$ and $d_U(w)(1-x_w)=0$ for every $w\in W$. 
Since $G$ has no isolated vertices, every vertex $w\in W$ has Perron coordinate $1$ and is adjacent to every vertex of $U$. Therefore, 
the set $W\cup\{u_*\}$ is independent, $W\cup\{u_*\}$ is complete to $U$, and the only possible remaining edges are the edges inside $U$.

If $e(U)=0$, then $G$ is complete bipartite, as desired. We now consider the case that $e(U)\ge 1$.  Every edge of $U$ has all vertices of $W\cup\{u_*\}$ as common neighbors, so
\[
|W| +1 \le r.
\]  
Using the Perron equations on the vertex $u_*$, we get
\[ \lambda^2= 
\lambda^2 x_{u_*}=\sum_{v\in U} \sum_{w\in N(v)} x_w 
=|U| (|W| +1)+ \sum_{uv\in E(U)} (x_u+x_v).
\]
Since $m=|U| (|W| +1)+ e(U)$ and 
 $\lambda^2=m$, it follows that
\begin{equation} \label{eq:eU}
\sum_{uv\in E(U)} (x_u+x_v)=e(U).
\end{equation}
Next, we will multiply $\lambda$ both sides of (\ref{eq:eU}). 
For every $u\in U$, we have
\[
\lambda \, x_u=|W| +1 +\sum_{z\in N_U(u)}x_z.
\]
Therefore, we get from (\ref{eq:eU}) that 
\begin{equation} \label{eq:lambda-mul} 
 \lambda \sum_{uv\in E(U)}\bigl( x_u+ x_v\bigr) 
=
e(U) \cdot 2(|W| +1)+
\sum_{uv\in E(U)}\left(\sum_{z\in N_U(u)}x_z+
\sum_{z\in N_U(v)}x_z\right).
\end{equation}
The last double sum can be regrouped as
\[
\sum_{uv\in E(U)}\bigl(d_U(u)x_v+d_U(v)x_u\bigr) \le r\sum_{uv\in E(U)}(x_u+x_v).
\]
where the last inequality holds since $d_U(u), d_U(v)\le r$. 
Thus, the previous (\ref{eq:lambda-mul}) yields 
\[
\lambda \sum_{uv\in E(U)}(x_u+x_v) \le e(U)\cdot 2(|W| +1)+r \sum_{uv\in E(U)}(x_u+x_v).
\]
Using (\ref{eq:eU}) and $e(U)\ge 1$, the above   simplifies to 
$\lambda \le 2(|W| +1) +r \le 3r$ since $|W|+1\le r$. 
This contradicts the assumption $\lambda(G)\ge 4r$. Hence, the case $e(U)\ge 1$ does not occur. 
\end{proof}

\subsection{Proof of Theorem \ref{thm-C4+}}

\label{sec:kites}

First, we show that the constant $\frac{1}{8}$ in Theorem \ref{thm-C4+} is best possible.

\begin{example} \label{exam-sharp-constant}
Let $s,t$ be positive integers with $t\to \infty$ and $s=4t$.
Let $G_t$ be obtained from $K_{s,t}$ by adding an edge inside   
the vertex part of size $s$. Then 
$G_t$ is Nosal and 
$$\texttt{\#}  C_4^+(G_t)= 
\left(\frac{1}{8}- o(1) \right)e(G_t). $$
\end{example}

\begin{proof}
Clearly, we have 
$      e(G_t)=st+1=4t^2+1.$ 
Let \(S,T\) be the two parts of \(K_{s,t}\), where \(|S|=s\) and \(|T|=t\).
Consider the unit vector \( \bm{y}\) defined by
\[
        y_v:=\frac{1}{\sqrt{2s}}\quad(\text{if}~v\in S),
        \qquad
        y_v:=\frac{1}{\sqrt{2t}}\quad(\text{if}~v\in T).
\]
The Rayleigh quotient gives
\[
\begin{aligned}
        \lambda(G_t)
        \ge
        \bm{y}^T A(G_t)\bm{y}                               =
        \bm{y}^T A(K_{s,t}) \bm{y}+\frac1s                   =
        \sqrt{st}+\frac1s
        =
        2t+\frac{1}{4t}.
\end{aligned}
\]
Therefore, we have 
\[
        \lambda(G_t)^2
        \ge
        \left(2t+\frac{1}{4t}\right)^2
        =
        4t^2+1+\frac{1}{16t^2}
        >
        4t^2+1
        =
        e(G_t).
\]
Thus \(G_t\) is a Nosal graph.

Finally, the only edge of \(G_t\) with at least two common neighbors is the
added edge inside \(S\), and its common neighbors are precisely all the \(t\)
vertices of \(T\).  Hence, we get 
\[
        \texttt{\#}{C_4^+}(G_t)
        =
        {t\choose2}
        =
        \left(\frac18 - o(1)\right)e(G_t).
\]
Thus, the constant \(1/8\) in Theorem \ref{thm-C4+} cannot be improved. 
\end{proof}

We summarize the structure of the lower bound of Theorem \ref{thm-C4+}  in Figure~\ref{fig:kite-proof}.

\begin{figure}[htbp]
\centering
\begin{tikzpicture}[
    node distance=5mm and 12mm,
    box/.style={rectangle, rounded corners, draw, align=center,
                text width=66mm, inner sep=3pt, font=\footnotesize},
    arr/.style={-{Stealth[length=1.8mm]}, semithick}
  ]

  \node[box] (assume)
        {Setup: Assume for contradiction that $\lambda^2>m$ but $\texttt{\#}\,C_4^+(G)\le(\tfrac18-\varepsilon)\lambda^2$};
  \node[box, below=of assume] (perron)
        {Perron decomposition: take $u_*$ with $x_{u^*}=1$, set $U=N(u_*)$ and $W=V\setminus(\{u_*\}\cup U)$};
  \node[box, below=of perron] (tools)
        {Two general tools (Theorems~\ref{thm:identity} and \ref{thm:halfmissing}): 
       Surplus identity: $\lambda^2 -m= \sum_{uv} 
       (x_u+x_v -1) - \Phi$
   \newline Half missing-weights:  $\Phi \ge \frac{1}{2}\sum_{u\in U} x_u Y(u)$ };
  \node[box, below=of tools] (ord)
        {Split bad edges of $U$ into ordinary and exceptional; ordinary edges are absorbed: $\sum_{\mathrm{ord}}(x_u+x_v-1)\le(1-2\delta+o(1))\Phi$};
  \node[box, below=of ord] (dec)
        {If $\mathcal E=\emptyset$, i.e., all bad edges ordinary, then $\lambda^2-m< 0$, a contradiction. Hence $\mathcal E\neq\emptyset$};

  \node[box, right=of assume, anchor=north,
        at={($(assume.north)+(66mm,0)$)}] (exc)
        {An exceptional edge must exist: $1\le |\mathcal E|=O(\lambda)$ and $\Phi=O(\lambda)$, so missing weights between $U$ and $W$ is $O(\lambda)$ (Claim~\ref{cl:big-O})};
  \node[box, below=of exc] (abc)
        {Fix an exceptional edge $uv \in E(U)$;  
         denote \\
          $A\!= \!\!\!\!\sum\limits_{z\in U,z\sim u} \!\!\! x_z,~
        B \!=\!\!\!\!\sum\limits_{z\in U,z\sim v} \!\!\! x_z ,~
        C\!= \!\!\!\!\!\sum\limits_{w\in W,w\sim u,v} \!\!\!\!\! x_w$. Then $A+B+2C\ge\lambda-O(1)$};
  \node[box, below=of abc] (count)
        {Establish local kite-count: $\texttt{\#}\,C_4^+\ge\tfrac12 \big(A^2+B^2+C^2 \big)+AC+BC-O(\lambda)$ (Claim~\ref{cl:lower-copies})};
  \node[box, below=of count] (min)
        {Minimizing the above under $A+B+2C \ge \lambda -O(1)$ yields $\texttt{\#}\,C_4^+ \ge\tfrac18\lambda^2-O(\lambda)$};
  \node[box, below=of min] (contra)
        {Contradiction; the sharp constant $\tfrac18$ is proved};

  \draw[arr] (assume)--(perron);
  \draw[arr] (perron)--(tools);
  \draw[arr] (tools)--(ord);
  \draw[arr] (ord)--(dec);

  \draw[arr] (exc)--(abc);
  \draw[arr] (abc)--(count);
  \draw[arr] (count)--(min);
  \draw[arr] (min)--(contra);

  \draw[arr] (dec.east) -| ($(dec.east)!0.5!(abc.west)$) |- (exc.west);
\end{tikzpicture}
\caption{The flowchart of the lower bound of Theorem~\ref{thm-C4+}.}
\label{fig:kite-proof}
\end{figure}
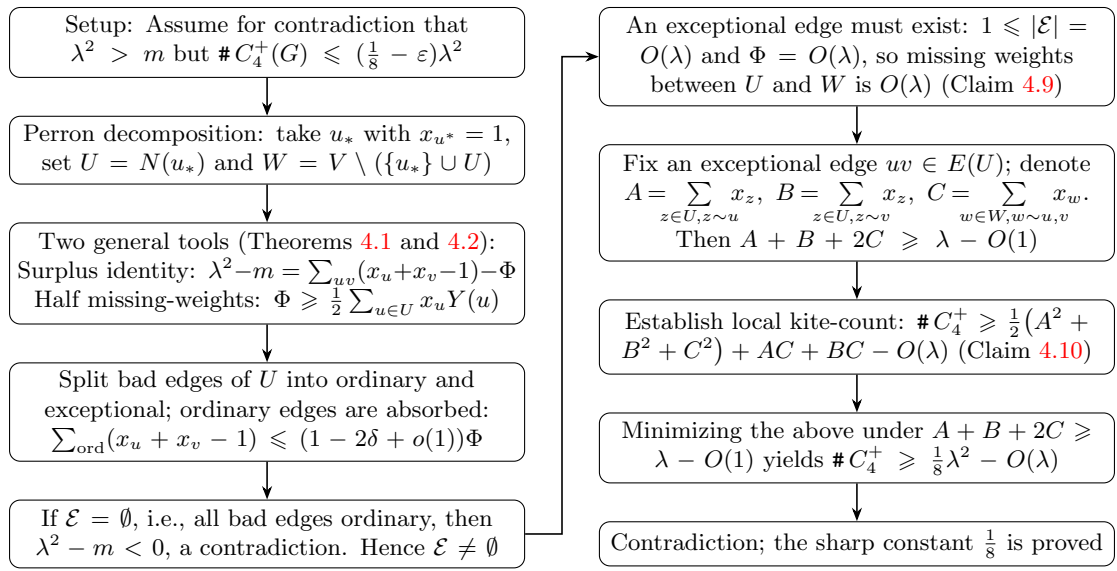

\smallskip 
Denote $  \tau(uv):=|N(u)\cap N(v)|$ for every  $uv\in E(G)$. 
Then
\begin{equation}\label{eq:kite-book-count}
        \texttt{\#}{C_4^+}(G)=\sum_{uv\in E(G)}{\tau(uv)\choose 2}.
\end{equation}  
For an edge \(uv\in E(U)\), we denote their  common neighbors in $W$ by 
$ Q_{uv}:=N_W(u)\cap N_W(v)$ 
and their common weights in $W$ by 
$C_{uv}:=\sum_{w\in Q_{uv}}x_w$. 
Moreover, we denote 
\[
        L_u(uv):=1+\sum_{z\in N_U(u)}x_z +C_{uv},
        \qquad
        L_v(uv):=1+ \sum_{z\in N_U(v)}x_z +C_{uv}.
\]
Recall that for each $u\in U$, the missing-weights of $u$ in $W$ is denoted by 
$$ Y(u) :=\sum_{w\in W\setminus N(u)}x_w. $$ 

\begin{lemma} \label{lem:Yu-Yv-lower-weighted}
    We have $Y(u)\ge \bigl(\lambda x_v-L_v(uv)\bigr)_+$ and $Y(v)\ge \bigl(\lambda x_u-L_u(uv)\bigr)_+$. 
\end{lemma}

\begin{proof} 
Applying the Perron equation at the vertex \(v\), we get
\[
        \lambda x_v
        =
        1+\sum_{z\in N_U(v)}x_z+\sum_{w\in N_W(v)}x_w .
\]
Therefore
\[
\begin{aligned}
        Y(u)     
        \ge
        \sum_{w\in N_W(v)\setminus N_W(u)}x_w     =
        \sum_{w\in N_W(v)}x_w-\sum_{w\in Q_{uv}}x_w  =
        \lambda x_v-1- \sum_{z\in N_U(v)}x_z -C_{uv}.
\end{aligned}
\]
Thus, we have 
\begin{equation*}
        Y(u)\ge \bigl(\lambda x_v-L_v(uv)\bigr)_+.
\end{equation*}
By symmetry, we get 
\begin{equation*}
        Y(v)\ge \bigl(\lambda x_u-L_u(uv)\bigr)_+. \qedhere 
\end{equation*} 
\end{proof}

To prove Theorem \ref{thm-C4+}, 
we shall use the following elementary inequality.

\begin{lemma}\label{lem:kite-algebra}
Let \(a,b\in[0,1]\), and let \(\alpha,\beta\in[0,1)\) satisfy
\(\alpha+\beta\le1\).  If \(a+b>1\), then
\[
        a+b-1
        \le
        \frac{a(b-\beta)_+}{1-\beta}
        +
        \frac{b(a-\alpha)_+}{1-\alpha},
\]
where we use the notation \((t)_+:=\max\{t,0\}\).
\end{lemma}

\begin{proof}
There are three cases.  If \(b\le \beta\), then
\(a>1-b\ge 1-\beta\ge\alpha\), so only the second positive part is needed.
We have
\[
\frac{b(a-\alpha)}{1-\alpha}-(a+b-1)
=
\frac{(1-a)(1-\alpha-b)}{1-\alpha}\ge0,
\]
because \(b\le\beta\le 1-\alpha\).  The case \(a\le\alpha\) is symmetric.

It remains to consider the case \(a>\alpha\) and \(b>\beta\).  In this case, we can write
\[
        a=\alpha+(1-\alpha)s,
        \qquad
        b=\beta+(1-\beta)t,
\]
where \(s,t\in(0,1]\).  
A direct calculation gives
\[
\begin{aligned}
\frac{a(b-\beta)}{1-\beta}
+
\frac{b(a-\alpha)}{1-\alpha}
-(a+b-1)                              =
        st+ (1-\alpha-\beta)(1-s)(1-t)\ge0.
\end{aligned}
\]
This proves the lemma.
\end{proof}

Now, we are going to prove Theorem \ref{thm-C4+}. 

\begin{proof}[{\bf Proof of Theorem \ref{thm-C4+}}] 
Let $ \lambda=\lambda(G)$. 
Since $  \lambda^2\le 2m$, 
the assumptions \(\lambda^2>m\) and \(m\to\infty\) imply
\(\lambda\to\infty\), and conversely \(\lambda\to\infty\) implies
\(m\to\infty\). We need to show 
$$ 
\texttt{\#}\,C_4^+ (G)\ge \left( 
\frac{1}{8}-o(1) \right)m.$$ 
Suppose to the contrary that the lower bound fails.  Then, after decreasing
\(\varepsilon\) if necessary, there is a fixed
\(0<\varepsilon<1/8\) and a sequence of \(m\)-edge graphs with \(m\to\infty\)
such that
\begin{equation*}
        \lambda^2>m
        \qquad\text{and}\qquad
        \texttt{\#}\,C_4^+ (G) \le \left(\frac18-\varepsilon\right)m .
\end{equation*} 
Since \(m<\lambda^2\), we also have
\begin{equation}\label{eq:kite-K-lambda-upper}
        \texttt{\#}\,C_4^+ (G) \le \left(\frac18-\varepsilon\right)\lambda^2.
\end{equation}

We call an edge \(uv\in E(U)\) \emph{bad} if
$   x_u+x_v>1$. 
Only bad edges can contribute positively to the sum in Theorem \ref{thm:identity}. 
We divide bad edges within $U$ into two classes.  A bad edge \(uv\in E(U)\) is called
\emph{ordinary} if
$   L_u(uv)+L_v(uv)\le \lambda$; 
and \emph{exceptional} otherwise.

We next prove a local coefficient estimate for ordinary bad edges.

\begin{claim}\label{lem:kite-local-coeff}
For every fixed \(\varepsilon>0\), there exists
\(\delta=\delta(\varepsilon)>0\) such that, under the assumption
\eqref{eq:kite-K-lambda-upper}, for every \(u\in U\), we have 
\[
        \sum_{\substack{v\in N_U(u)\\ uv\ {\rm ordinary\ bad}}}
        \frac{1}{\lambda-L_v(uv)}
        \le
        \frac12-\delta+o(1).
\]
\end{claim}

\begin{proof}[Proof of claim]
First, note that the edge \(u_*u\) has \(d_U(u)\) common neighbors.  Hence, by
\eqref{eq:kite-book-count},
\[
     {d_U(u)\choose2} \le \texttt{\#}\,C_4^+ (G).
\]
Together with \eqref{eq:kite-K-lambda-upper}, this implies that
$   d_U(u)\le \left(\frac12-\gamma+o(1)\right)\lambda$ 
for some \(\gamma=\gamma(\varepsilon)>0\).

Fix a small constant \(\xi>0\), to be chosen later. We split the ordinary bad
neighbors \(v\) of \(u\) into two sets:
\[
        I_0=\{v:L_v(uv)\le \xi\lambda\},
        \qquad
        I_1=\{v:L_v(uv)>\xi\lambda\}.
\]
For every \(v\in I_0\), we have 
$  \frac{1}{\lambda-L_v(uv)}
        \le
        \frac{1}{(1-\xi)\lambda}$. 
Therefore, we get 
\[
        \sum_{v\in I_0}\frac{1}{\lambda-L_v(uv)}
        \le
        \frac{d_U(u)}{(1-\xi)\lambda}
        \le
        \frac{\frac12-\gamma+o(1)}{1-\xi} \le   \frac12-\frac{\gamma}{2}+o(1),
\]
where the last inequality holds by 
choosing \(\xi>0\) sufficiently small in terms of \(\varepsilon\). 

Now consider \(v\in I_1\).  Since
$   L_v(uv)=1+ \sum_{z\in N_U(v)}x_z +C_{uv}\le 1+d_U(v)+|Q_{uv}|$, 
we have
\[
        d_U(v)+|Q_{uv}| \ge \xi\lambda-O(1).
\]
For notational convenience, we denote 
$ d:=d_U(v)$ and  $q:=|Q_{uv}|$. 
The two center edges \(u_*v\) and \(uv\) give two distinct terms in the
kite-counting sum \eqref{eq:kite-book-count}; namely, they contribute at least
$  {d\choose2}+{q+1\choose2}$ 
copies of $C_4^+$.  Since \(d+q\ge\xi\lambda-O(1)\), we have
\[
        {d\choose2}+{q+1\choose2}
        \ge
        \frac14(d+q)^2-O(d+q)
        \ge
        c_\xi\lambda^2
\]
for all sufficiently large \(\lambda\), where \(c_\xi>0\) depends only on
\(\xi\).  For every $v\in I_1$, all the center edges \(u_*v\) and \(uv\) are
distinct.  Hence summing over \(v\in I_1\) gives
$  |I_1|c_\xi\lambda^2\le  \texttt{\#}{C_4^+}(G)$.
Thus \(|I_1|=O_{\varepsilon,\xi}(1)\).

Moreover, for each \(v\in I_1\), the inequality
\[
        {d\choose2}+{q+1\choose2}
        \le  \texttt{\#}{C_4^+}(G)
        \le
        \left(\frac18-\varepsilon\right)\lambda^2
\]
first gives \(d,q=O(\lambda)\), and then gives
\[
        d^2+q^2
        \le
        \left(\frac14-2\varepsilon+o(1)\right)\lambda^2.
\]
Therefore, by Cauchy--Schwarz's inequality, we obtain 
\[
        d+q
        \le
        \sqrt{2(d^2+q^2)}
        \le
        \kappa\lambda+o(\lambda)
\]
for some constant \(\kappa=\kappa(\varepsilon)<0.8\).  Hence
\[
        L_v(uv)\le 1+d+q\le \kappa\lambda+o(\lambda).
\]
Consequently every term $\frac{1}{\lambda - L_v(uv)}$ with \(v\in I_1\) is \(O_\varepsilon(1/\lambda)\).
Since \(|I_1|=O_{\varepsilon,\xi}(1)\), we get
\[
        \sum_{v\in I_1}\frac{1}{\lambda-L_v(uv)}=o(1).
\]
Combining the estimates for \(I_0\) and \(I_1\), the claim follows by setting \(\delta=\gamma/4\).
\end{proof}

We now show that all ordinary bad edges are absorbed by the deficit term \(\Phi\) and  $G$ must contain an exceptional edge. 
Let \(uv\) be an ordinary bad edge.  In Lemma~\ref{lem:kite-algebra}, setting 
\[
        a=x_u,\qquad b=x_v,\qquad
        \alpha=\frac{L_u(uv)}{\lambda},\qquad
        \beta=\frac{L_v(uv)}{\lambda}.
\]
Since \(uv\) is ordinary, we have \(\alpha+\beta\le1\). So Lemma~\ref{lem:kite-algebra} is applicable. 
Then 
\[ x_u+x_v -1 \le \frac{x_u (\lambda x_v - L_v(uv))_+}{\lambda - L_v(uv)} + \frac{x_v(\lambda x_u -L_u(uv))_+}{\lambda - L_u(uv)}.   \]
Combining with 
Lemma \ref{lem:Yu-Yv-lower-weighted}, 
we get 
\[
        x_u+x_v-1
        \le
        \frac{x_uY(u)}{\lambda-L_v(uv)}
        +
        \frac{x_vY(v)}{\lambda-L_u(uv)}.
\]
Summing this inequality over all ordinary bad edges and applying
Claim~\ref{lem:kite-local-coeff}, we obtain
\begin{equation}\label{eq:kite-ordinary-absorbed}
        \sum_{\substack{uv\in E(U)\\ uv\ {\rm ordinary\ bad}}}
        (x_u+x_v-1)
        \le
        \left(\frac12-\delta+o(1)\right)
        \sum_{u\in U}x_uY(u)                      \le
        (1-2\delta+o(1))\Phi,
\end{equation}
where the last inequality follows from Theorem \ref{thm:halfmissing}.

\begin{claim}  \label{cl:big-O}
Let \(\mathcal E\) be the set of exceptional bad edges. 
    Then $\mathcal{E}\neq \emptyset$, $|\mathcal E|=O(\lambda)$ and $\Phi = O(\lambda)$. 
\end{claim}

\begin{proof}[Proof of claim]
 If \(\mathcal E=\emptyset\), i.e., all bad edges of $U$ are ordinary, 
then using 
Theorem \ref{thm:identity} and (\ref{eq:kite-ordinary-absorbed}) gives
$  \lambda^2-m
        \le
        -(2\delta-o(1))\Phi < 0$, 
a contradiction.  
Thus, we must have \(\mathcal E\neq\emptyset\). 

Let \(uv\in\mathcal E\) be an exceptional edge.  Then
$  L_u(uv)+L_v(uv)>\lambda$, 
and therefore
\[
        2+\sum_{z\in N_U(u)}x_z + 
        \sum_{z\in N_U(v)}x_z 
        +2C_{uv}>\lambda.
\]
Since \( \sum_{z\in N_U(u)}x_z \le d_U(u)\), \( \sum_{z\in N_U(v)}x_z \le d_U(v)\), and \(C_{uv}\le |Q_{uv}|\), the above gives
\begin{equation}\label{eq:kite-exceptional-size}
        d_U(u)+d_U(v)+2|Q_{uv}| \ge \lambda-O(1).
\end{equation}

First, we consider exceptional edges with \( |Q_{uv}|\ge\lambda/4\).  For each such
edge, the center edge \(uv\) contributes at least
$  {|Q_{uv}|+1\choose2}
        \ge
        \left(\frac1{32}-o(1)\right)\lambda^2$ 
copies of \(C_4^+\).  Since the center edges \(uv\) are distinct, and there are at most $(\frac{1}{8} - \varepsilon)\lambda^2$ copies of $C_4^+$, 
there are
only \(O(1)\) such exceptional edges.

Now, we consider exceptional edges with \(|Q_{uv}| <\lambda/4\).  By
\eqref{eq:kite-exceptional-size}, we have 
\[
        d_U(u)+d_U(v)\ge \frac12\lambda-O(1).
\]
Hence, the number of such edges is at most
\[
        \frac{2+o(1)}{\lambda}
        \sum_{uv\in E(U)}(d_U(u)+d_U(v))
        =
        \frac{2+o(1)}{\lambda}
        \sum_{u\in U}d_U(u)^2.
\]
But the center edge \(u_*u\) contributes \({d_U(u)\choose2}\) copies of
\(C_4^+\), so
\[
        \sum_{u\in U}{d_U(u)\choose2}\le \texttt{\#} C_4^+(G)\le \left(\frac{1}{8} - \varepsilon \right)\lambda^2.
\]
Consequently, we get 
\[
        \sum_{u\in U}d_U(u)^2
        \le
        2\sum_{u\in U}{d_U(u)\choose2} +\sum_{u\in U}d_U(u)
        =  
        O(\lambda^2).
\]
Thus, we obtain that the number of exceptional edges with \(|Q_{uv}| <\lambda/4\) is
\(O(\lambda)\).  Therefore, we conclude that the total number of exceptional edges $ |\mathcal E|=O(\lambda)$. 

We now return to the spectral surplus identity.  Edges in \(E(U)\) that are not
bad have nonpositive contribution.  Hence, using 
\eqref{eq:kite-ordinary-absorbed} and $|\mathcal E| =O(\lambda)$, we have 
\[
\begin{aligned}
        \lambda^2-m
        &=
        \sum_{uv\in E(U)}(x_u+x_v-1)-\Phi                         \\
        &\le
        \sum_{uv\in\mathcal E}(x_u+x_v-1)
        +(1-2\delta+o(1))\Phi-\Phi                                \\
        &\le
        O(\lambda)-(2\delta-o(1))\Phi.
\end{aligned}
\]
Since \(\lambda^2-m>0\), it follows that
$\Phi=O(\lambda)$, as needed. 
\end{proof}

By Theorem \ref{thm:halfmissing} and Claim \ref{cl:big-O}, we get
\begin{equation}\label{eq:kite-missing-Olambda}
        \sum_{\substack{u\in U,\ w\in W\\ uw\notin E(G)}}x_ux_w
        =
        \sum_{u\in U}x_uY(u)
        \le 2\Phi =
        O(\lambda).
\end{equation}

 Fix an exceptional bad
edge \(uv\in\mathcal E\), we write 
$ Q=Q_{uv}=N_W(u)\cap N_W(v)$ 
and 
\[
        A=\sum_{z\in N_U(u)}x_z,
        \qquad
        B=\sum_{z\in N_U(v)}x_z,
        \qquad
        C=\sum_{w\in Q}x_w.
\]
Since \(uv\) is exceptional, we have 
\begin{equation}\label{eq:kite-exceptional-ABC}
        A+B+2C\ge \lambda-O(1).
\end{equation}

We now count the copies of \(C_4^+\) forced by this exceptional edge $uv\in E(U)$.  

\begin{claim} \label{cl:lower-copies}
    We have $\texttt{\#}\, C_4^+(G) \ge
        \frac12(A^2+B^2+C^2)+AC+BC-O(\lambda) $. 
\end{claim}

\begin{proof}[Proof of claim]
First,
the three center edges \(u_*u\), \(u_*v\), and \(uv\) give
\[
        {d_U(u)\choose2}
        +
        {d_U(v)\choose2}
        +
        {|Q|+1\choose2}
\]
copies of $C_4^+$.  Under the contradiction assumption \eqref{eq:kite-K-lambda-upper},
the quantities \(d_U(u)\), \(d_U(v)\), and \(|Q| \) are all \(O(\lambda)\).
Since \(d_U(u)\ge A\), \(d_U(v)\ge B\), and \(|Q|\ge C\), these three center
edges contribute at least
\begin{equation}\label{eq:kite-three-centers}
        \frac12(A^2+B^2+C^2)-O(\lambda)
\end{equation}
copies of \(C_4^+\).

The crucial additional contribution comes from two rectangles.  Consider first
the pairs
\[
        (z,w)\in \bigl(N_U(u)\setminus\{v\}\bigr)\times Q.
\]
If \(zw\in E(G)\), then the center edge \(uw\) has two common neighbors
\(v\) and \(z\).  Hence the choice of such a pair \((z,w)\) gives a copy of
\(C_4^+\) with center edge \(uw\).  Therefore the number of such copies is at
least
\[
        \left|\{(z,w): z\in N_U(u)\setminus\{v\},\ w\in Q,\ zw\in E(G)\}\right|.
\]
Since every Perron coordinate is at most \(1\), this number is at least the
corresponding weighted count:
\[
        \sum_{\substack{z\in N_U(u)\setminus\{v\}\\ w\in Q\\ zw\in E(G)}}x_zx_w.
\]
The total weighted mass of all pairs in
\(\bigl(N_U(u)\setminus\{v\}\bigr)\times Q\) is \((A-x_v)C\).  By
\eqref{eq:kite-missing-Olambda}, the total weighted mass of missing pairs
between \(U\) and \(W\) is \(O(\lambda)\).  Moreover,
\[
        C\le \sum_{w\in N_W(u)}x_w\le \lambda x_u\le \lambda
\]
by the Perron equation at \(u\).  Hence
\begin{equation}\label{eq:kite-rect-u} 
        \sum_{\substack{z\in N_U(u)\setminus\{v\}\\ w\in Q\\ zw\in E(G)}}x_zx_w
        \ge
        (A-x_v)C-O(\lambda) \ge
        AC-O(\lambda). 
\end{equation}
Similarly, the rectangle
$   \bigl(N_U(v)\setminus\{u\}\bigr)\times Q$ 
gives at least
\begin{equation}\label{eq:kite-rect-v}
        BC-O(\lambda)
\end{equation}
additional copies of \(C_4^+\) with center edges \(vw\), where \(w\in Q\).

The center edges used in \eqref{eq:kite-three-centers},
\eqref{eq:kite-rect-u} and \eqref{eq:kite-rect-v} are pairwise distinct:
they are respectively \(u_*u,u_*v,uv\), then edges \(uw\) with \(w\in Q\), and
edges \(vw\) with \(w\in Q\).  Thus, the corresponding contributions are 
added in the kite-counting formula \eqref{eq:kite-book-count}.  We obtain
\begin{equation*}  
        \texttt{\#}\, C_4^+(G) 
        \ge
        \frac12(A^2+B^2+C^2)+AC+BC-O(\lambda). 
        \qedhere 
\end{equation*}
\end{proof}

It remains to minimize $\texttt{\#}\, C_4^+(G)$ in Claim \ref{cl:lower-copies} under the constraint 
\eqref{eq:kite-exceptional-ABC}. We denote 
$   T:=A+B+2C $.   
Since $A,B\ge 0$, we have \(A^2+B^2\ge \frac{1}{2}(A+B)^2\), and then 
\[
\begin{aligned}
        \frac12(A^2+B^2+C^2)+AC+BC
        &\ge
        \frac{(A+B)^2}{4}+\frac{C^2}{2}+(A+B)C.
\end{aligned}
\]
Using \(A+B=T-2C\), the right-hand side becomes
\[
        \frac{(T-2C)^2}{4}+\frac{C^2}{2}+(T-2C)C
        =
        \frac{T^2}{4}-\frac{C^2}{2} \ge \frac{T^2}{8},
\]
where the last inequality holds since \(A+B\ge0\) and \(0\le C\le T/2\).  

By \eqref{eq:kite-exceptional-ABC}, we know that \(T\ge\lambda-O(1)\).  Hence
\[
        \frac12(A^2+B^2+C^2)+AC+BC
        \ge
        \frac18\lambda^2-O(\lambda).
\]
Together with Claim \ref{cl:lower-copies}, this yields
\[
      \texttt{\#} \,C_4^+(G)  \ge \frac18\lambda^2-O(\lambda)
        =
        \left(\frac18-o(1)\right)\lambda^2, 
\]
which contradicts \eqref{eq:kite-K-lambda-upper}. This completes the proof of the lower
bound. 
\end{proof}

\subsection{Counting books of arbitrary size} 

\label{sec:large-books}

Recently, 
Li, Lin, Liu and Zhang \cite{LLLZ2026} established the edge-spectral 
supersaturation results by developing the spectral Sidorenko inequalities. For example, recall that $S_{t-1,m}$ is the split graph with $m$ edges obtained by joining a clique $K_{t-1}$ with an independent set. It was shown in \cite{LLLZ2026} that for every fixed $t\ge 2$ and every $m$-edge graph $G$ with $\lambda(G)>\lambda(S_{t-1,m})$, we have 
$$\texttt{\#} K_{t,t}(G)
        \ge
        \left(\frac{2^{-(t-1)^2}}{(t!)^2}-o(1)\right)m^t \quad \text{and}\quad
 \texttt{\#}C_{2t}(G)
        \ge
        \left(\frac{(t-1)!}{2t^t}-o(1)\right)m^t.$$
The extremal constructions of the first are random graphs, while the second are split graphs. 

It is natural to ask a general problem for counting the book $B_t$ of arbitrary size $t$.

\begin{theorem}\label{thm:Bt-count}
For any fixed $t\ge 1$, 
every $m$-edge graph $G$ with 
$\lambda(G)>\sqrt m$ satisfies
\[
   \texttt{\#}\,B_t(G)\ \ge\ \left(\frac{1}{t!\,2^{t}}-o(1)\right)m^{t/2}. 
\]
\end{theorem}

The constant $\tfrac{1}{t!\,2^{t}}$ is best possible by the construction in Example \ref{exam-sharp-constant}. Let $G$ be the graph obtained from a complete bipartite graph $K_{4k,k}$ by adding one edge inside the partite set of size $4k$. As proved in Example \ref{exam-sharp-constant}, we have $\lambda (G)> \sqrt{m}$, where $m=4k^2+1$. 
Every book in $G$ has its spine on the added edge, whose endpoints share exactly $k$ common neighbors, so $\texttt{\#}B_t(G)=\binom{k}{t}$. Setting $k\to \infty$  yields $\texttt{\#} B_t(G)=\bigl(\tfrac{1}{t!\,2^{t}}-o(1)\bigr)m^{t/2}$. 
 Theorem \ref{thm:Bt-count} and this construction  imply  
\[
   \lim_{m\to\infty}\ \min_{\lambda(G)>\sqrt m}\
   \frac{\texttt{\#}\,B_t(G)}{m^{t/2}}\;=\;\frac{1}{t!\,2^{t}} .
\]
 We will show that the approach of Theorems \ref{thm:book-free} and \ref{thm-C4+} remains valid  for Theorem \ref{thm:Bt-count} by combining the spectral surplus identity together with the half-missing-weights estimate. 

Observe that the $t=1$ case of Theorem \ref{thm:Bt-count} reduces to count the number of triangles, and it can be proved by applying  Ning--Zhai's spectral theorem \cite{NingZhai2023}. 
 For every $t\ge2$, 
 every copy of the book $B_t$ is uniquely determined by its spine
together with $t$ of the common neighbors of the endpoints of the spine. Hence, by denoting $\tau(uv):=|N(u)\cap N(v)|$, we have 
\begin{equation}\label{eq:book-count-t}
   \texttt{\#}\,B_t(G)=\sum_{uv\in E(G)}\binom{\tau(uv)}{t}.
\end{equation}
This directly extends the kite-count formula used in the proof of Theorem~\ref{thm-C4+}.

\begin{proof}[{\bf Proof of Theorem~\ref{thm:Bt-count}}]
Write $\lambda=\lambda(G)$; as in the proof of Theorem~\ref{thm-C4+} the
hypotheses are equivalent to $\lambda\to\infty$.  Suppose for contradiction that
the bound fails.  Then there are a fixed $\varepsilon\in(0,0.01)$
and a sequence of $m$-edge graphs with $m\to\infty$ such that
\begin{equation}\label{eq:Bt-contra}
   \lambda^2>m
   \qquad\text{and}\qquad
   \texttt{\#}\,B_t(G)\le\Bigl(\frac{1}{t!\,2^{t}}-\varepsilon\Bigr)\lambda^{t}. 
\end{equation}
 We use repeatedly that
$\binom{d}{t}=\tfrac{d^t}{t!}+O(d^{t-1})$ for every $0\le d\le\lambda$.

\medskip
\noindent\textbf{Step 1. An exceptional edge exists.}
This step is identical to the corresponding part of the proof of
Theorem~\ref{thm-C4+}.  The weighted bounds
in Lemma \ref{lem:Yu-Yv-lower-weighted} and the elementary inequality in Lemma~\ref{lem:kite-algebra} are
independent of $t$.  The local coefficient estimate persists:

\begin{claim}\label{cl:Bt-coeff}
Under \eqref{eq:Bt-contra}, there is $\delta=\delta(\varepsilon)>0$ such that 
for every $u\in U$,
\[
   \sum_{\substack{v\in N_U(u)\\ uv\ \mathrm{ordinary\ bad}}}
   \frac{1}{\lambda-L_v(uv)}\ \le\ \frac12-\delta+o(1).
\]
\end{claim}

\begin{proof}[Proof of claim]
The proof is similar to that of Claim~\ref{lem:kite-local-coeff} with two
substitutions.  First, using $\binom{d_U(u)}{t}\le\texttt{\#}\,B_t(G)$ and
\eqref{eq:Bt-contra} give $d_U(u)\le(\tfrac12-\gamma)\lambda$ for some
$\gamma=\gamma(\varepsilon)>0$.  Second, for $v\in I_1$ (with $d=d_U(v)$ and 
$q=|Q_{uv}|$), one has $\binom{d}{t}+\binom{q+1}{t}\le\texttt{\#}\,B_t(G)
\le(\tfrac1{t!2^t}-\varepsilon)\lambda^t$, whence
$d^t+q^t\le(\tfrac1{2^t}-t!\,\varepsilon+o(1))\lambda^t$, and the
Cauchy--Schwarz step is replaced by the power-mean inequality
\[
   d+q\ \le\ 2^{\,1-1/t}\bigl(d^t+q^t\bigr)^{1/t}\ \le\ 2^{-1/t}\lambda\,(1+o(1)),
\]
with $2^{-1/t}<1$.  Thus $L_v(uv)\le\kappa\lambda+o(\lambda)$ for some
$\kappa=\kappa(\varepsilon)<1$, and each such term is $O(1/\lambda)$.  All other
steps are unchanged.
\end{proof}

Summing the inequality of Lemma~\ref{lem:kite-algebra} over ordinary bad edges
and applying Claim~\ref{cl:Bt-coeff} and Theorem~\ref{thm:halfmissing} yields,
exactly as in \eqref{eq:kite-ordinary-absorbed},
\begin{equation}\label{eq:Bt-ordinary}
   \sum_{\substack{uv\in E(U)\\ uv\ \mathrm{ordinary\ bad}}}(x_u+x_v-1)
   \le(1-2\delta+o(1))\Phi .
\end{equation}
If the set $\mathcal E$ of exceptional edges were empty, then
Theorem~\ref{thm:identity} would give
$$\lambda^2-m\le-(2\delta-o(1))\Phi < 0, $$  contradicting $\lambda^2>m$.  Hence, we must have 
$\mathcal E\neq\emptyset$.

\medskip
\noindent\textbf{Step 2. The deficit $\Phi$ is $O(\lambda)$.}
In the setting of $t\ge 3$, the argument of Claim~\ref{cl:big-O} must be modified.  
That proof bounds 
$\sum_{u\in U}d_U(u)^2$ through $\sum_u\binom{d_U(u)}{2}\le
\texttt{\#}\,C_4^+(G)=O(\lambda^2)$. However, for $t\ge3$,  no such bound on the number of
kites is available, so we need to argue differently.

\begin{claim}\label{cl:Bt-deficit}
Under \eqref{eq:Bt-contra}, we have $|\mathcal E|=O(\lambda)$ and $\Phi=O(\lambda)$.
\end{claim}

\begin{proof}[Proof of claim]
For $uv\in\mathcal E$, we have, exactly as in \eqref{eq:kite-exceptional-size},
\begin{equation}\label{eq:Bt-exc-size}
   d_U(u)+d_U(v)+2|Q_{uv}|\ \ge\ \lambda-O(1).
\end{equation}

First, we 
consider the {exceptional edges $uv\in E(U)$ with $|Q_{uv}|\ge\lambda/4$.}  
The spine $uv$ has
 at least $|Q_{uv}|\ge\lambda/4$ common neighbors in $W$, so it leads to at least
$\binom{\lfloor\lambda/4\rfloor}{t}\ge(\tfrac1{t!\,4^t}-o(1))\lambda^t$ copies
of $B_t$ on distinct spines. By \eqref{eq:Bt-contra}, we have $\texttt{\#}\,B_t(G)\le\bigl(\frac{1}{t!\,2^{t}}-\varepsilon\bigr)\lambda^{t}$. So the number of exceptional edges with $|Q_{uv}| \ge \lambda /4$ is at most
$ 2^t + o(1)= O(1)$. 

Second, we count the exceptional edges $uv\in E(U)$ with $|Q_{uv}|<\lambda/4$.  
By \eqref{eq:Bt-exc-size}, we have 
$$ \max\{d_U(u),d_U(v)\}\ge \frac\lambda4-O(1).$$ 
Let
$H:=\{y\in U:\ d_U(y)\ge\tfrac\lambda4-O(1)\}$.  Each vertex $y\in H$ gives a spine
$u_*y$ of codegree $d_U(y)\ge\lambda/4$, which leads to at least
$(\tfrac1{t!\,4^t}-o(1))\lambda^t$ copies of $B_t$ on distinct spines. Since  $\texttt{\#}\,B_t(G)\le\bigl(\frac{1}{t!\,2^{t}}-\varepsilon\bigr)\lambda^{t}$ by \eqref{eq:Bt-contra}, we get 
$|H|\le2^t+o(1)=O(1)$. 
Since every exceptional edge with $|Q_{uv}|<\lambda/4$
has an endpoint in $H$, and each vertex $y\in H$ lies in at most $d_U(y)\le\lambda$
edges of $E(U)$, which is guaranteed by using (\ref{eq:Bt-contra}) again, we obtain that there are at most $|H|\cdot\lambda=O(\lambda)$ such edges.
Consequently, we conclude that $|\mathcal E|=O(\lambda)$, as needed. 

Finally, since non-bad edges of $E(U)$ contribute at most $0$,
Theorem~\ref{thm:identity} and \eqref{eq:Bt-ordinary} give
\begin{align*}
  \lambda^2-m &= \sum_{uv\in E(U)} (x_u+x_v -1) - \Phi \\ 
   &\le\sum_{uv\in\mathcal E}(x_u+x_v-1)+(1-2\delta+o(1))\Phi-\Phi \\
   & \le |\mathcal E|-(2\delta-o(1))\Phi,
\end{align*}
Since $ \lambda^2- m >0$ and $|\mathcal{E}| = O(\lambda)$, we get $\Phi=O(\lambda)$, as needed. 
\end{proof}

Using the half-missing-weights estimate of  Theorem~\ref{thm:halfmissing} and Claim~\ref{cl:Bt-deficit}, it follows that 
\begin{equation}\label{eq:Bt-missing}
   \sum_{\substack{u\in U,\ w\in W\\ uw\notin E(G)}}x_ux_w
   =\sum_{u\in U}x_uY(u)\le 2\Phi=O(\lambda).
\end{equation}

\medskip
\noindent\textbf{Step 3. Local count and conclusion.}
Fix an exceptional edge $uv\in\mathcal E$ and denote 
$Q:=Q_{uv}$, $A:=\sum_{z\in N_U(u)}x_z$, $B:=\sum_{z\in N_U(v)}x_z$,
$C:=\sum_{w\in Q}x_w$.  Then $A,B,C\le\lambda$ and, as in
\eqref{eq:kite-exceptional-ABC},
\begin{equation}\label{eq:Bt-ABC}
   T:=A+B+2C\ \ge\ \lambda-O(1).
\end{equation}
For $t=2$, the required count came in Claim~\ref{cl:lower-copies} from three
center edges and two rectangles.  For general $t\ge 3$, the center edges alone no
longer suffice, and the key estimate is the following.

\begin{claim}  
\label{lem:Bt-rect} 
If $A\to\infty$, then
$\texttt{\#}\,B_t(G)\ge\bigl(C-O(\lambda/A)\bigr)\binom{\lfloor A/2\rfloor}{t}$.
The symmetric statement with $A,u$ replaced by $B,v$ also holds.
\end{claim}

\begin{proof}[Proof of claim]
For each $w\in Q$, we denote 
$$ a_w=\sum_{z\in N_U(u),\,zw\in E(G)}x_z\le A. $$ 
Reversing the
order of summation and using \eqref{eq:Bt-missing}, we get 
\[
   \sum_{w\in Q}x_w a_w
   =\sum_{z\in N_U(u)}x_z\!\!\sum_{w\in Q,\,zw\in E}\!\!x_w
   =\sum_{z\in N_U(u)}x_z\Bigl(C-\!\!\sum_{w\in Q,\,zw\notin E}\!\!x_w\Bigr)
   \ge AC-O(\lambda).
\]
Let $Q':=\{w\in Q:\ a_w\ge A/2\}$ and $C_1:=\sum_{w\in Q'}x_w$.  Since $a_w\le A$
on $Q'$ and $a_w<A/2$ off $Q'$,
$$ 
AC-O(\lambda)\le \sum_{w\in Q} x_w a_w \le  A\,C_1+\frac A2(C-C_1). 
$$ 
Hence, we have 
$\tfrac A2 C_1\ge\tfrac A2 C-O(\lambda)$ and $C_1\ge C-O(\lambda/A)$. Since each
weight $x_w$ is at most $1$, we obtain $|Q'|\ge C_1\ge C-O(\lambda/A)$.  For each vertex $w\in Q'$, the
common neighbors of $u$ and $w$ contain every $z\in N_U(u)$ adjacent to $w$, so
$$ \tau(uw)\ge |\{z\in N_U(u):zw\in E\}| \ge a_w\ge A/2. $$ 
Then the spine $uw$ carries
at least $\binom{\lfloor A/2\rfloor}{t}$ copies of $B_t$. Since the spines $uw$
(with $w\in Q'\subseteq W$) are pairwise distinct, by \eqref{eq:book-count-t}, 
it follows that 
\[
   \texttt{\#}\,B_t(G)\ge\sum_{w\in Q'}\binom{\tau(uw)}{t}
   \ge|Q'|\binom{\lfloor A/2\rfloor}{t}
   \ge\bigl(C-O(\lambda/A)\bigr)\binom{\lfloor A/2\rfloor}{t}. 
\]
This completes the proof of Claim \ref{lem:Bt-rect}. 
\end{proof}

Now, we return to the proof of Theorem~\ref{thm:Bt-count}.   
Choose $\delta_1=\delta_1(\varepsilon)\in(0,1)$ so
small that $(1-\delta_1)^t>1-t!\,2^t\varepsilon$,
and choose a constant
$C_0=C_0(\varepsilon,\delta_1)\ge 2(2/\delta_1)^t$
large enough to absorb the implicit constants in the $O(\cdot)$ terms below.
 We split into three cases according to the size of
$C$, and then we derive a contradiction with \eqref{eq:Bt-contra} in each case.

\subsection*{Case 1:  
$C\ge\tfrac\lambda2(1-\delta_1)$}   

The spine $uv$ has codegree
$\tau(uv)\ge|Q|+1\ge C$, so by \eqref{eq:book-count-t},
\[
   \texttt{\#}\,B_t(G)\ge\binom{\lceil C\rceil}{t}\ge\frac{C^t}{t!}(1-o(1))
   \ge\frac{(1-\delta_1)^t}{t!\,2^t}\lambda^t(1-o(1))
   >\Bigl(\frac1{t!2^t}-\varepsilon\Bigr)\lambda^t,
\]
by the choice of $\delta_1$, which leads to a contradiction.

\medskip 
In the remaining cases $C<\tfrac\lambda2(1-\delta_1)$, so by \eqref{eq:Bt-ABC}, we have 
$A+B\ge\delta_1\lambda-O(1)$. Relabelling $u,v$ if necessary, we may assume that 
$A\ge\tfrac{\delta_1}{2}\lambda-O(1)=\Theta(\lambda)$.

\subsection*{Case 2: $C\le C_0$} 

Then $A+B\ge\lambda-2C_0-O(1)=\lambda-O(1)$, and
the spines $u_*u$, $u_*v$ give
\[
   \texttt{\#}\,B_t(G)\ge\binom{\lceil A\rceil}{t} \!+ \! \binom{\lceil B\rceil}{t}
   \ge\frac{A^t+B^t}{t!}-O(\lambda^{t-1})
   \ge\frac{(A+B)^t}{t!\,2^{t-1}}-O(\lambda^{t-1})
   \ge\frac{\lambda^t}{t!\,2^{t-1}}(1-o(1)),
\]
where we used the convexity bound $A^t+B^t\ge(A+B)^t/2^{t-1}$.  This contradicts
\eqref{eq:Bt-contra}.

\subsection*{Case 3: $C_0<C<\tfrac\lambda2(1-\delta_1)$}   

Here
$A\ge\tfrac{\delta_1}{2}\lambda$, so $A\to\infty$, $O(\lambda/A)=O(1)$ and
$C-O(\lambda/A)\ge C/2$.  By Claim~\ref{lem:Bt-rect},
\[
   \texttt{\#}\,B_t(G)\ge\frac{C}{2}\binom{\lfloor A/2\rfloor}{t}
   \ge\frac{C}{2}\cdot\frac{(\delta_1\lambda/4)^t}{t!}(1-o(1)).
\]
Since $C>C_0=2(2/\delta_1)^t$, we have
$\tfrac C2(\tfrac{\delta_1}{4})^t>\tfrac1{2^t}$. So the right-hand side in the above exceeds
$\tfrac{\lambda^t}{t!\,2^t}(1-o(1)) 
>(\tfrac1{t!2^t}-\varepsilon)
\lambda^t$ for sufficiently large $\lambda$, again a contradiction.

All three cases lead to a contradiction with 
\eqref{eq:Bt-contra}. This completes the proof. 
\end{proof}

\begin{remark}
For $t=2$, the above argument is an alternative to the proof in Section~\ref{sec:4}:
Case~1 alone reproduces the constant $\tfrac18$, while Cases~2--3 show that
any configuration with $A$ or $B$ of order $\lambda$ carries strictly more
kites.  For $t\ge3$, the modification in Step~2 (Claim~\ref{cl:Bt-deficit}) and
the count estimate in Step~3 (Claim \ref{lem:Bt-rect}) are essential, since the kite-specific fact 
$\texttt{\#}\,C_4^+(G)=O(\lambda^2)$ and the exact lower bound of Claim~\ref{cl:lower-copies} are no longer available when we count $\texttt{\#}B_t$ for each $t\ge 3$.
\end{remark}

\section{Concluding remarks}
\label{sec:concluding}

In this paper, we proved several edge-spectral supersaturation results for triangles and books, which give spectral counterparts and refinements of classical  results. 
For triangles, Theorem~\ref{thm-edge-spect-LS} shows 
that $\lambda(G)\ge\sqrt{m} + q$ forces more
than $q\, m$ triangles, giving an edge-spectral counterpart of the Lov\'asz--Simonovits theorem. Theorem~\ref{thm:strong} refines this by establishing the bound $t(G)\ge m(\lambda-\sqrt{m}\,)$, with equality only for complete bipartite graphs. 
This improves the classical Bollob\'{a}s--Nikiforov bound \cite{BollobasNikiforov2007} in the range $\sqrt{m} \le \lambda \le  1.3\sqrt{m}$.  
For books, our contribution contains three parts. Theorem~\ref{thm:book-free} is a
spectral Tur\'an-type theorem: every $B_{r+1}$-free graph with $m\ge(4r)^2$ edges
satisfies $\lambda(G)\le\sqrt{m}$, again with the complete bipartite graphs as
the only extremal examples. 
The other two results describe what happens once
$\lambda(G)$ rises above this bound. 
Theorem~\ref{thm-size-14} shows that every Nosal graph contains a book of size more than $\tfrac14\sqrt{m}$, making progress toward the   bound $\tfrac13\sqrt{m}$ proposed by Li, Liu and Zhang~\cite{LiLiuZhang26}, and toward a conjecture of Li, Feng and Peng \cite{LFP2024-triangular} on counting triangular edges. 
Theorem~\ref{thm-C4+} says that every Nosal graph contains at least
$(\tfrac18-o(1))m$ copies of the kite $C_4^+=B_2$, with the best possible constant $\tfrac18$. 
In Theorem \ref{thm:Bt-count}, 
we established a sharp extension $\texttt{\#}\, B_t(G) \ge (\frac{1}{t!2^t} - o(1))m^{t/2}$ for all $t\ge 1$, which unifies both triangles and kites.    

\medskip 
We close with some related problems and possible directions for future work. 

\vspace{-3mm}
\paragraph{Substructures with layer-by-layer jump phenomenon.}
The triangle count jumps in discrete layers 
as $\lambda(G)$ exceeds successive
thresholds: the first layer $\lambda (G)>\sqrt m$ forces $t(G)\ge 
\lfloor \frac{1}{2}(\sqrt{m} -1) \rfloor$,
the second layer $\lambda (G) \ge \tfrac{1}{2}(1+\sqrt{4m-3}\,)$ forces
$t(G)\ge\tfrac{1}{2}(m-1)$, and the third layer $\lambda (G)\ge 1+ \sqrt{m-2}$  forces $t(G)\ge m-2$ (Theorem \ref{thm:m-2}). 
It is natural to ask whether the kite count
$\texttt{\#}\,C_4^+$ follows the same layer-by-layer jumps. At the first layer, 
the kite-count is $ (\frac{1}{8} - o(1))m$ (Theorem~\ref{thm-C4+}), and one might expect it to jump to 
$(\frac{1}{8} - o(1))m^2$ across the second layer as seen by $K_2\vee \frac{m-1}{2}K_1$; however, this expectation fails. Let 
$G$ be obtained from $K_{n,n}$ by adding a cycle $C_n$ inside one part. A short computation shows that $G$
has $m=n^2+n$ edges and spectral radius $\lambda(G)=1+\sqrt{1+n^2}$, which satisfies 
$\lambda(G)\ge \frac{1}{2}(1+ \sqrt{4m-3}\,)$, that is, $G$ lies above the second layer. On the other hand, the
only edges of $G$ with at least two common neighbors are the $n$  edges of $C_n$, each lying on
a spine with the $n$ vertices of the opposite part, 
together with the $n^2$ cross edges of $K_{n,n}$,
each of codegree $2$. Hence $\texttt{\#}\,C_4^+(G)=n\binom{n}{2}+n^2=\tfrac12 n^2(n+1)
=\Theta(m^{3/2})$.  
Thus, the kite count does not jump in the same way as the triangle count that takes the split graphs as references,
and determining its true behavior at each layer remains an open problem. 
A broader problem is to determine other graphs whose supersaturations have the layer-by-layer jump behavior. 

\paragraph{An exact formula for $\bm{t(m,\lambda)}$.}
Let $r: = \lambda /\sqrt m$ be the spectral density of $G$. Then $1\le r<\sqrt2$ and the inequality \eqref{eq:spectral-lower} obtained
in the proof of Theorem~\ref{thm:strong} reads
\begin{equation} \label{eq-single}
   t(G)\ge \frac16 \left(\lambda^3-(2m-\lambda^2)^{3/2}
   \right)= \frac16 m^{3/2}\left(r^3-(2-r^2)^{3/2}\right).
\end{equation}
This bound is in fact stronger than both the
Bollob\'as--Nikiforov bound
$\tfrac13\lambda(\lambda^2-m)$ and the linear bound
$m(\lambda-\sqrt m\,)$ of Theorem~\ref{thm:strong}
throughout the range $\sqrt m\le\lambda<\sqrt{2m}$.
Indeed, setting $s:=\sqrt{2-r^2}$, a short computation gives
$ \tfrac16\bigl(r^3-s^3\bigr)
   -\tfrac13 r(r^2-1)=\tfrac16 s^2(r-s)\ge 0$, 
with equality only at $r=1$ and $r=\sqrt2$; 
the linear bound is dominated as well, since $\tfrac16(r^3-s^3)\ge r-1$ on the same range, 
with equality only at $r=1$. 
Thus, the lower
bound (\ref{eq-single}) 
is stronger throughout $\sqrt{m}\le \lambda < \sqrt{2m}$. This suggests the following spectral extremal problem.

\begin{problem}
In the full range $\sqrt m\le\lambda<\sqrt{2m}$,
determine the exact formula
\[
   t(m,\lambda):=\min\bigl\{\,t(G):e(G)=m,\
   \lambda(G)=\lambda\,\bigr\}.
\]
\end{problem}

We next prove that the main term of $t(m,\lambda)$ is given by  
\[
   \lim_{m\to\infty}\frac{t(m,r\sqrt m\,)}{m^{3/2}}
   =\frac16 \left(r^3-(2-r^2)^{3/2}\right).
\] 
In view of (\ref{eq-single}), 
it suffices to construct a graph asymptotically
 attaining this bound. 
Taking the split graph $G=K_{a+1}\vee tK_1$, we have $m=e(G)=\binom{a+1}{2}+(a+1)t$. Moreover, 
one checks that 
\[
  \lambda (G)=\frac{a+\sqrt{a^2+4(a+1)t}}{2}\quad \text{and} \quad 
  t(G)=\binom{a\!+\!1}{3}+\binom{a\!+\!1}{2}t.
\]
The edge count gives
$4(a+1)t=4m-2a^2-2a$, so 
$ \lambda (G) =\frac{1}{2}\big(a+\sqrt{4m-a^2-2a}\, \big)$. Setting $a:=c\sqrt m$ for a constant $c$  to be determined, we get 
${\lambda(G)}/{\sqrt m} =\frac12\big(c+\sqrt{4-c^2 
-{2c}/\!{\sqrt m}}\, \big)
  \to \frac{1}{2} \big(c+\sqrt{4-c^2} \,\big)$ as $m\to \infty$. 
Solving the equation $r=\frac{1}{2}\big( c+\sqrt{4-c^2}\, \big)$ yields $c=r-\sqrt{2-r^2} =r-s$.  

Eliminating $t$, we get $t(G)=\frac{1}{6}(a^3 -a) + \frac{1}{2}am - \frac{1}{4}(a^3 +a^2)$. Using $a=c\sqrt{m}$ and $c=r-s$, we obtain $t(G)/ m^{3/2} \to \frac{c}{2} - \frac{c^3}{12} = \frac{1}{6}(r^3-s^3)$ as $m\to \infty$, matching the bound above. 

At $\lambda=\sqrt m$, this gives $r=s=1$, so $a=0$ and the above sharp construction reduces to $K_{1,m}$; while as $\lambda\to\sqrt{2m}$, it gives $a\to\sqrt{2m}$ and the construction is close to the clique $K_{\sqrt{2m}}$. We refer to
\cite{LPS2020} for the exact minimum number of
triangles in graphs with given edge density.

In addition, a companion problem concerns the graph 
stability. Since the equality in
Theorem~\ref{thm:strong} forces $G$ to be complete
bipartite, one expects a stability version: Suppose
that $G$ satisfies $\lambda(G)\ge\sqrt m+q$ and
$t(G)\le q\,m+o(m^{3/2})$. How close must $G$ be to a
complete bipartite graph?

\paragraph{Closing the gap for the booksize.}
Theorem~\ref{thm-size-14} gives $\bk(G)>\tfrac14\sqrt m$ for every $m$-edge Nosal graph $G$. 
 In our proof, 
the constant $\tfrac14$ comes exactly from  the threshold $\lambda\ge 4r$ during the proof of Theorem~\ref{thm:book-free}, which is in turn dictated by the elementary inequality of Lemma \ref{lem:algebra}, which together with Lemma~\ref{lem:Yu-lower} converts every bad edge $uv$ of $E(U)$ into the missing weights, that is, $x_u+x_v-1\le \frac{1}{2r}\big( x_uY(u)+x_vY(v) \big)$. 
Summing over all edges of $E(U)$, and combining with the half-missing weights estimate, it follows that $\sum_{uv\in E(U)} (x_u+x_v -1)\le \Phi$ and $\lambda^2 \le m$. 

\smallskip 
Li, Liu and Zhang~\cite{LiLiuZhang26} proposed the following problem; recall from the introduction that the constant $\frac13$ cannot be replaced by any larger one as seen by the blowup of the $3$-prism.

\begin{problem}[Li--Liu--Zhang \cite{LiLiuZhang26}]
Every Nosal graph $G$ satisfies $\bk(G)>\tfrac13\sqrt m$.    
\end{problem}

Lowering the threshold from $4r$ to $3r$ in Theorem~\ref{thm:book-free} would deliver the conjectured constant $\frac{1}{3}$. The construction of the triangular-prism blow-up described in the introduction is $B_{r+1}$-free, has $(9-o(1))r^2$ edges, and still satisfies $\lambda(G)>\sqrt m$. 
The obstacle lies in the regime $3r\le\lambda<4r$, where a bad edge of $U$ can no longer be absorbed by the deficit $\Phi$ alone, so one must instead exploit the finer local structure of the induced subgraph $G[U]$ forced by $B_{r+1}$-freeness, e.g., in the spirit of the ordinary/exceptional dichotomy in the proof of Theorem~\ref{thm-C4+}.

 \paragraph{Joints and generalized books.} 
An {\it $r$-joint} is a collection of $r$-cliques sharing a common edge. 
Let $js_r(G)$ denote the maximum number of $r$-cliques in an $r$-joint of $G$. The study of  joints can be traced back to the works of Erd\H{o}s \cite{Erd1969}, and Bollob\'{a}s and Nikiforov \cite{BN2008, BN2011}.  
The {\it generalized book}  $B_{r,k}:=K_r\vee I_k$ is obtained by joining every vertex of a clique $K_r$ to every vertex of an independent set $I_k$ of size $k$. Equivalently, $B_{r,k}$ consists of $k$ cliques $K_{r+1}$ sharing a common $K_r$. We refer to \cite{NR2004, LP2021, CFW2022} for related results. 
Both joints and generalized books extend the classical book. 

\smallskip 
Using the probabilistic method, Li, Liu and Zhang \cite{LiLiuZhang26} proved that for any $r\ge 2$, every $m$-edge graph $G$ 
 with $\lambda^2(G)> (1- \frac{1}{r}) 2m$ satisfies $js_{r+1}(G)=\Omega_r (m^{(r-1)/2})$. Invoking the Kruskal--Katona theorem,  then $G$ contains a copy of $B_{r,k}$ of size $k=\Omega_r(\sqrt{m})$, and these bounds are tight up to constant factors. 
This gives an extension of Nikiforov's results for cliques \cite{Nikiforov2002} and books \cite{Nikiforov21}. 
It is interesting to determine the exact sharp bounds on joints and generalized books.

\begin{problem}
For every fixed integer $r\ge 3$, what are the largest constants $C_r$ and $C_r'$ such that for all sufficiently large $m$, if $G$ is an $m$-edge graph  with 
$$\lambda^2(G) > \Big(1- \frac{1}{r} \Big)2m,$$ 
then $G$ satisfies $js_{r+1}(G)\ge C_r m^{(r-1)/2}$, and $G$ contains a copy of  $B_{r,k}$ with $k= C_r' \sqrt{m}$? 
\end{problem}

\paragraph{Counting other substructures.} 
Ning and Zhai~\cite{NingZhai2023} showed that every graph $G$ with 
$\lambda(G) > \sqrt{m}$ contains at least $\lfloor \frac{1}{2}(\sqrt{m}-1) \rfloor$ 
triangles, and Li, Liu and Zhang~\cite{LiLiuZhang26} proved that the same condition 
forces at least $(\frac{1}{8} - o(1))\,m^2$ copies of $C_4$. It is natural to ask 
the analogous question for $5$-cycles, where the spectral extremal bound is known: 
Zhai, Lin and Shu \cite{ZLS2021} showed that a $C_5$-free graph $G$ with $m \ge 8$ edges satisfies 
$\lambda(G) \le \frac{1}{2}(1+\sqrt{4m-3}\,)$, with equality if and only if 
$G = K_2 \vee \frac{m-1}{2}K_1$. Does there exist a constant $C>0$ such that every 
graph $G$ of large size $m$ with $\lambda(G) > \frac{1}{2}(1+\sqrt{4m-3}\,)$ 
contains at least $Cm^2$ copies of $C_5$, and if so, what is the largest such $C$? 
An affirmative answer would make the order $m^2$ best possible, as witnessed by $K_3 \vee \frac{m-3}{3}K_1$. 

\medskip 
Recently, Fang, Lin and  Zhai~\cite{FangLinZhai}  established the edge-spectral supersaturation 
for all color-critical graphs with chromatic number at least $4$, 
and Li, Lin, Liu and Zhang~\cite{LLLZ2026} treated the classical bipartite graphs with the Sidorenko property, including $K_{t,t}$ and $C_{2t}$. 
Two cases thus remain open: color-critical graphs with chromatic number $3$ 
(e.g., $C_{2t+1}$ and $K_{t,t}^+$) and bipartite graphs whose Sidorenko property 
is unknown (e.g., $K_{5,5}\setminus C_{10}$).

\section*{Acknowledgments}
The authors developed the arguments, wrote and verified all proofs presented in this paper. During an early exploratory stage, the authors additionally used language-model-based tools to brainstorm candidate  strategies. Any suggestions arising from these tools served only as informal inspiration.

\end{document}